\documentclass[11pt]{amsart}
\usepackage{amscd}
\usepackage{amsfonts}
\usepackage[all]{xy}
\usepackage{amsmath,amsthm,hyperref}
\usepackage{amsmath,amssymb,amsthm,latexsym}
\usepackage{amscd}
\usepackage{amsmath}
\numberwithin{equation}{section}
\usepackage{setspace}
\usepackage{caption}
\usepackage{xcolor}
\newtheorem{theorem}{Theorem}[section]

\newtheorem{definition}[theorem]{Definition}
\newtheorem{corollary}[theorem]{Corollary}
\newtheorem{lemma}[theorem]{Lemma}

\theoremstyle{remark}
\newtheorem{remark}[theorem]{Remark}

\newcommand{\R}{\mathbb{R}}
\newcommand{\C}{\mathbb{C}}

\begin{document}
\title[On totally isotropic Willmore two-spheres in $S^6$]{\bf{ Willmore surfaces in spheres via loop groups IV:  on totally isotropic Willmore two-spheres in $S^6$}}
\author{Peng Wang}
\maketitle

\begin{center}
{\bf Abstract}

\end{center}
Totally isotropic surfaces in $S^6$ are not necessarily Willmore surfaces.
Therefore it is  the  first goal of this paper to derive a geometric characterization of totally isotropic Willmore two-spheres in $S^6$. This will naturally yield to a description of such surfaces in terms of the loop group language.
Moreover, applying the loop group method, we also obtain an algorithm to construct totally isotropic Willmore two-spheres in $S^6$.
As an application, a new, totally isotropic Willmore two-sphere which is not S-Willmore (i.e., has no dual surface) in $S^6$ is constructed,   illustrating the theory of this paper.
\ \\



\tableofcontents

\section{Introduction}

Totally isotropic surfaces first appeared in the study of  the global geometry of surfaces in the famous work of Calabi \cite{Calabi}, where twistor bundle theory was applied to describe the geometry of minimal two-spheres in $S^n$. This led later to much progress in geometry and the theory of integrable systems (see for example \cite{Bryant1982} and  \cite{BR}).

A basic characterization of totally isotropic surfaces in $S^{2n}$ is that
they correspond to projections of holomorphic/anti-holomorphic curves in the twistor bundle  $\mathfrak{T}S^{2n}$ of  $S^{2n}$ \cite{Calabi}, \cite{Bryant1984}, \cite{Ejiri1988}, \cite{BR}, \cite{Mus1}, \cite{Mon}. And a well-known description of (branched) minimal two-spheres in $S^{2n}$ is that
they correspond to projections of horizontal holomorphic/anti-holomorphic curves in the twistor bundle  $\mathfrak{T}S^{2n}$ of  $S^{2n}$ \cite{Calabi}, \cite{Bryant1984}, \cite{BR}, \cite{BuGu}.

In the study of Willmore two-spheres, totally isotropic surfaces play an important role as well. First we note that isotropic properties are conformally invariant. This indicates that they are of interest in the conformal geometry of surfaces. Moreover, the classical work of Ejiri \cite{Ejiri1988} shows that isotropic surfaces in $S^4$ are automatically Willmore surfaces and furthermore they are Willmore surfaces with dual surfaces. He also showed that Willmore two-spheres in $S^4$ are either M\"{o}bius equivalent to minimal surfaces with planer ends in $\R^4$, or are isotropic two-spheres \cite{Ejiri1988} (See also \cite{Mus1}, \cite{Mon} and \cite{BFLPP}). This generalization of Bryant's seminal work on Willmore two-spheres \cite{Bryant1984}, \cite{Bryant1988}, shows the complexity of Willmore two-spheres in $S^n$, when $n>3$.

In \cite{Ejiri1988}, Ejiri also introduced the notion of S-Willmore surfaces. Roughly speaking, these surfaces can be viewed as Willmore surfaces for which there exist dual surfaces.  Note that by Bryant's classical work, every Willmore surface in $S^3$ has a dual Willmore surface \cite{Bryant1984}. But when the codimension is bigger than $1$, a Willmore surface may not have a  dual surface \cite{Ejiri1988} (See \cite{BFLPP}, and \cite{Ma2006} for the generalizations along this direction). Using the duality properties of S-Willmore surfaces, Ejiri  provided furthermore a classification of S-Willmore two-spheres in $S^{n+2}$ by constructing the holomorphic forms for these surfaces \cite{Ejiri1988}. Especially,
for Willmore two-spheres in $S^4$, a construction of a holomorphic 8-form indicates that these surfaces automatically  are
S-Willmore \cite{Ejiri1988}, \cite{Mus1}, \cite{Mon}, \cite{BFLPP}, \cite{Ma}.

In the end of Ejiri's paper he conjectured that all Willmore two-spheres in $S^{n+2}$ are S-Willmore.
 If  Ejiri's  conjecture is incorrect, that is, if some Willmore, but not S-Willmore two-spheres would exist, how would one  construct and characterize them?

In this paper,  we will give an answer to this question: we will construct a totally isotropic Willmore two-sphere in $S^6$ which is not S-Willmore. This answers  Ejiri's  conjecture in  an explicit  way. Moreover,  beyond the explicit construction of some new examples it is a main goal of this paper is to characterize all totally isotropic Willmore two-spheres in $S^6$ via their geometric properties and their normalized potentials.  This geometric description also  supplies the basis for the work  of \cite{Wang-1}, where we provide a coarse classification of Willmore two-spheres in spheres by using the loop group method for the construction of harmonic maps \cite{DPW}, \cite{BuGu}, \cite{Gu2002}, \cite{DoWa2}.

We point out  that, different from the case in $S^4$, where totally isotropic surfaces are automatically S-Willmore surfaces  and of finite uniton type, totally isotropic surfaces in $S^6$ are not even Willmore in general. (We refer to \cite{Uh}, \cite{BuGu}, \cite{DoWa2} and \cite{Wang-1} for the definition of uniton.) Moreover,  even  a totally isotropic Willmore surface in $S^6$ will, in general,  not be of finite uniton type. See Remark 2.5. \\

This paper is organized as follows.
In Section 2, we first recall some basic results about Willmore surfaces, and derive a new geometric description of isotropic Willmore two-spheres in $S^6$. Applying this description, and also collecting some basic loop theory, we obtain a description of the normalized potentials of totally isotropic Willmore two-spheres in $S^6$. The converse part, i.e., that generically such normalized potentials will always produce special totally isotropic Willmore surfaces of finite uniton type in $S^6$, as well as new examples, makes up the main content of Section 3. The main idea is to perform a concrete Iwasawa decomposition for the  normalized potentials under consideration to obtain the geometric properties of the corresponding Willmore surfaces. This also yields an algorithm to construct examples of Willmore surfaces. Illustrating this algorithm, we also construct concrete totally isotropic Willmore two-spheres in $S^6$. Since the computations are too technical, we put the computations of Iwasawa decompositions and  examples into two Appendixes.

\section{Isotropic Willmore two-spheres in $S^6$}
This section has two parts. In the first part, we will recall the basic theory about Willmore surfaces  and then focus our attention on isotropic Willmore surfaces in $S^6$. Restricting to two-spheres, by construction of holomorphic forms, we derive a description of the Maurer-Cartan form of isotropic Willmore two spheres in $S^6$. In the second part, we will recall the basic DPW methods as well as Wu's formula for harmonic maps in symmetric spaces. Then by the geometric description, applying Wu's formula, we will derive the form of the normalized potentials for the harmonic conformal Gauss map of isotropic Willmore two-spheres in $S^6$.

\subsection{Isotropic Willmore surfaces in $S^6$ and related holomorphic differentials}

\subsubsection{Willmore surfaces in spheres}
For completeness we first recall briefly the basic surface theory. For more details, we refer to Section 2 of \cite{DoWa1} and Section 2 of \cite{Wang-1} (see also \cite{BPP}, \cite{Ma}).

Let $\mathbb{R}^{n+4}_1$ be the Lorentz-Minkowski space equipped with the Lorentzian metric
\[\langle x,y\rangle=-x_{0}y_0+\sum_{j=1}^{n+3}x_jy_j=x^t I_{1,n+3} y,\ \  I_{1,n+3}=diag\left(-1,1,\cdots,1\right), \ \forall
  x,y\in\R^{n+4}.\]
We denote by $\mathcal{C}^{n+3}_+= \lbrace x \in \mathbb{R}^{n+4}_{1} |\langle x,x\rangle=0 , x_0 >0 \rbrace $
the forward light cone and by $Q^{n+2}=\mathcal{C}^{n+3}_+/ \R^+$  the  projective light cone. It is well-known that Riemannian space forms can be conformally embedded into $Q^{n+2}$ \cite{Bryant1984}, \cite{Ejiri1988}, \cite{Mus1}, \cite{Mon}, \cite{{BFLPP}},  \cite{BPP}.

For a  conformal immersion $y:M\rightarrow S^{n+2}$ one has a canonical lift $Y=e^{-\omega}(1,y)$ into $\mathcal{C}^{n+3}$ with respect to a local complex coordinate $z$ of the Riemann surface $M$, where $e^{2\omega}=2\langle y_z,y_{\bar{z}}\rangle$.
There exists a global bundle decomposition
\[
M\times \mathbb{R}^{n+4}_{1}=V\oplus V^{\perp}, ~\ \hbox{ with }\ V={\rm Span}\{Y,{\rm Re}Y_{z},{\rm Im}Y_{z},Y_{z\bar{z}}\},
\]
where $V^{\perp}$ denotes the orthogonal complement of $V$.  Let $V_{\mathbb{C}}$ and
$V^{\perp}_{\mathbb{C}}$ be the complexifications of $V$ and $V^{\perp}$.
Let $\{Y,Y_{z},Y_{\bar{z}},N\}$  be a  frame of
$V_{\mathbb{C}}$ such that
$
\langle N,Y_{z}\rangle=\langle N,Y_{\bar{z}}\rangle=\langle
N,N\rangle=0,\ \langle N,Y\rangle=-1.
$ Let $D$
denote the normal connection on $V_{\mathbb{C}}^{\perp}$. For any section $\psi\in
\Gamma(V_{\mathbb{C}}^{\perp})$ of the normal bundle
and a canonical lift $Y$ w.r.t $z$, we obtain
the structure equations:
\begin{equation}\label{eq-moving}
\left\{\begin {array}{lllll}
 Y_{zz}=-\frac{s}{2}Y+\kappa,\ \\
Y _{z\bar{z}}=-\langle \kappa,\bar\kappa\rangle Y+\frac{1}{2}N,\ \\
 N_{z}=-2\langle \kappa,\bar\kappa\rangle Y_{z}-sY_{\bar{z}}+2D_{\bar{z}}\kappa,\ \\
 \psi_{z}=D_{z}\psi+2\langle \psi,D_{\bar{z}}\kappa\rangle Y-2\langle
\psi,\kappa\rangle Y_{\bar{z}}, \ \\
\end {array}\right.
\end{equation}
Here $\kappa$ is  \emph{the conformal Hopf differential} of $y$ ,
and $s$ is \emph{the Schwarzian} of $y$.
The conformal Gauss, Codazzi and Ricci equations as integrability conditions are
as follows:
\begin{equation}\label{eq-integ}
\left\{\begin {array}{lllll}
 \frac{1}{2}s_{\bar{z}}=3\langle
\kappa,D_z\bar\kappa\rangle +\langle D_z\kappa,\bar\kappa\rangle,\\
{\rm Im}(D_{\bar{z}}D_{\bar{z}}\kappa+\frac{\bar{s}}{2}\kappa)=0,\\
   R^{D}_{\bar{z}z}\psi=D_{\bar{z}}D_{z}\psi-D_{z}D_{\bar{z}}\psi =
  2\langle \psi,\kappa\rangle\bar{\kappa}- 2\langle
  \psi,\bar{\kappa}\rangle\kappa.
\end {array}\right.
\end{equation}
Recall that $y$ is a Willmore surface if and only if the Willmore equation holds (\cite{BPP})
\begin{equation}\label{eq-willmore}
D_{\bar{z}}D_{\bar{z}}\kappa+\frac{\bar{s}}{2}\kappa=0.
\end{equation}
Another equivalent condition of $y$ being Willmore is the harmonicity of the conformal Gauss map $Gr:M\rightarrow
Gr_{1,3}(\mathbb{R}^{n+4}_{1})=SO^+(1,n+3)/SO^+(1,3)\times SO(n)$ of $y$ \cite{Bryant1984}, \cite{Ejiri1988}, \cite{Ma}:
\[
Gr:=Y\wedge Y_{u}\wedge Y_{v}\wedge N=-2i\cdot Y\wedge Y_{z}\wedge
Y_{\bar{z}} \wedge N.
\]
A local lift of $Gr$ is chosen as \begin{equation}\label{F}
F:=\left(\frac{1}{\sqrt{2}}(Y+N),\frac{1}{\sqrt{2}}(-Y+N),e_1,e_2,\psi_1,\cdots,\psi_n\right)��  U\rightarrow  SO^+(1,n+3)
\end{equation}with its Maurer-Cartan form
 \[\alpha=F^{-1}dF=\left(
                   \begin{array}{cc}
                     A_1 & B_1 \\
                   -B_1^tI_{1,3} & A_2 \\
                   \end{array}
                 \right)dz+\left(
                   \begin{array}{cc}
                     \bar{A}_1 & \bar{B}_1 \\
                    -\bar{B}_1^tI_{1,3}& \bar{A}_2 \\
                   \end{array}
                 \right)d\bar{z},\]
\begin{equation}\label{eq-b1} B_1=\left(
      \begin{array}{ccc}
         \sqrt{2} \beta_1 & \cdots & \sqrt{2}\beta_n \\
         -\sqrt{2} \beta_1 & \cdots & -\sqrt{2}\beta_n \\
        -k_1 & \cdots & -k_n \\
        -ik_1 & \cdots & -ik_n \\
      \end{array}
    \right).   \end{equation} Here $\{\psi_j\}$ is an orthonormal basis of $V^{\perp}$ on $U$ and
  \[\kappa=\sum_j k_j\psi_j ,\ D_{\bar{z}}\kappa=\sum_j\beta_j\psi_j,\ k^2=\sum_j|k_j|^2.\]

\subsubsection{Isotropic Willmore surfaces in $S^6$}
Recall that $y$ is totally isotropic if and only if all the derivatives of $y$ with respect to $z$ are isotropic, that is,
\[\left\langle Y_z^{(m)}, Y_{z}^{(n)} \right\rangle =0 ~~\hbox{ for all }~m,\ n \in \mathbb{Z}^+.\] Here $Y_{z}^{(m)}$ means taking $m$ times derivatives of $Y$ by $z$.
We refer to \cite{Calabi}, \cite{Bryant1982}, \cite{Ejiri1988}, \cite{Ma} and \cite{DoWa1} for more discussions on isotropic surfaces. A well-known result states that $y$ is totally isotropic if and only if $y$ is the projection of a holomorphic or anti-holomorphic curve into the  twistor bundle $\mathfrak{T}S^{2n}$ of $S^{2n}$ \cite{Calabi}, \cite{Ejiri1988}. For the basic theory about twistor bundles, we refer to \cite{BR}.

Let $y$ be a Willmore surface with an isotropic Hopf differential, i.e., $\langle\kappa,\kappa\rangle\equiv0$. Note that one derives straightforwardly that
\[\langle\kappa, D_{z} \kappa\rangle=\langle\kappa, D_{\bar{z}} \kappa\rangle=0\] by differentiating $\langle\kappa,\kappa\rangle=0$. Applying the Willmore equation \eqref{eq-willmore}, we  also have \[\langle D_{\bar{z}}\kappa, D_{\bar{z}} \kappa \rangle\equiv 0.\]

For isotropic Willmore surfaces, Xiang Ma introduced several  holomorphic differentials, see Theorem 5.4 of \cite{Ma}. For our case, we only need that
\begin{equation}\label{eq-holo-d1}
\Omega dz^4:=\langle D_{\bar{z}}\kappa,D_{z} \kappa \rangle dz^4
\end{equation}
is a globally defined holomorphic differential on $M$. To show that
$\Omega dz^4$ is holomorphic is a direct computation by using  $\langle\kappa,\kappa\rangle=0$, Willmore equations and Ricci equations (see also \cite{Ma}). Then, if $M=S^2$, we will have that \[\langle D_{\bar{z}}\kappa, D_{z} \kappa \rangle\equiv 0.\]

 Now assume that $y$ is not S-Willmore, then $D_{\bar{z}}\kappa$ is not parallel to $\kappa$. So  $D_{\bar{z}}\kappa$ and $ \kappa$ span a two-dimensional isotropic subspace $\hbox{Span}_{\C}\{\kappa,D_{\bar{z}}\kappa\}$. Since $D_{z}\kappa$ is perpendicular to $\kappa$ and $D_{\bar{z}}\kappa$, $D_{z}\kappa$ is contained in $\hbox{Span}_{\C}\{\kappa,D_{\bar{z}}\kappa\}$. As a consequence, we also have $\langle D_{z}\kappa,D_{z}\kappa\rangle=0$.

Summing up, we obtain the following
\begin{theorem} \label{thm-iso-willmore-1}
Let $y$ be a Willmore two sphere in $S^6$ with isotropic Hopf differential, i.e., $\langle\kappa,\kappa\rangle=0$. If $y$ is not S-Willmore, then $y$ is totally isotropic (and hence full) in  $S^6$. Moreover, locally there exists an isotropic frame $\{E_1, E_2\}$ of the normal bundle $V^{\perp}_{\C}$ of $y$ such that
\begin{equation}\label{eq-normal-bundle}
\left\{\begin{split}
&\kappa,\ D_{z}\kappa,\ D_{\bar{z}}\kappa\in \hbox{Span}_{\C}\{E_1, E_2\},\\
&\langle E_i, E_j\rangle=0, \langle E_i, \bar{E}_j\rangle=\delta_{ij}, i,j=1,2.\\
&  D_{z} E_i\in \hbox{Span}_{\C}\{E_1, E_2\},\ D_{\bar{z}} E_i\in \hbox{Span}_{\C}\{E_1, E_2\},\ i,j=1,2.\\
                 \end{split}\right.
\end{equation}
That is, the normal connection is block diagonal under the frame $\{E_1, E_2, \bar{E}_1, \bar{E}_2 \}$.
\end{theorem}
This theorem can also be derived by loop group theory (see the end of subsection 2.2).

\ \\
Note that \eqref{eq-normal-bundle} provides also sufficient conditions for $y$ to be a Willmore surface, i.e., we have the following
\begin{theorem} \label{thm-iso-willmore-2}
Let $y$ be a totally isotropic surface from $U$ into $S^6$, with complex coordinate $z$. If there exists an isotropic frame $\{E_1, E_2\}$ of the normal bundle $V^{\perp}_{\C}$ of $y$ such that
 \eqref{eq-normal-bundle} holds, then $y$ is a Willmore surface.
\end{theorem}
\begin{proof}By \eqref{eq-normal-bundle}, we see that $D_{\bar{z}}D_{\bar{z}}\kappa+\frac{\bar{s}}{2}\kappa$ is an isotropic vector. Since $Im(D_{\bar{z}}D_{\bar{z}}\kappa+\frac{\bar{s}}{2}\kappa)=0$ by \eqref{eq-integ},  we have  $D_{\bar{z}}D_{\bar{z}}\kappa+\frac{\bar{s}}{2}\kappa=0$. So $y$ is Willmore.
\end{proof}

\subsection{Normalized potentials of totally isotropic Willmore two-spheres in $S^6$}

This subsection aims to derive the description of totally isotropic Willmore two-spheres  in $S^6$ in terms of the loop group methods.
To this end, we will first collect the basic theory concerning the DPW construction of harmonic maps and the applications to Willmore surfaces.
Then, we will derive the construction of normalized potentials of totally isotropic Willmore two-spheres via Wu's formula.
For more details of the loop group method we refer to \cite{DoWa1}, \cite{DoWa2} and \cite{Wu}.

\subsubsection{Harmonic maps into a symmetric space}
Let $G/K$ be a symmetric space defined by the involution $\sigma: G\rightarrow G$, with $G^{\sigma}\supset K\supset(G^{\sigma})_0$, and Lie algebras $\mathfrak{g}=Lie(G)$, $\mathfrak{k}=Lie(K)$. The Cartan decomposition induced by $\sigma$ on $\mathfrak{g}$ states that
\[ \mathfrak{g}=\mathfrak{k}\oplus\mathfrak{p}, \
[\mathfrak{k},\mathfrak{k}]\subset\mathfrak{k}, \
[\mathfrak{k},
\mathfrak{p}]\subset\mathfrak{p},
\ [\mathfrak{p},\mathfrak{p}]\subset\mathfrak{k}.
\]

 Let $f$ be a conformal harmonic map from a Riemann surface $M$ into
 $G/K$. Let $U$ be an open connected subset of $M$ with complex coordinate $z$. Then there exists a frame $F: U\rightarrow G$ of $f$ with a  Maurer-Cartan form $F^{-1} d F= \alpha$. The Maurer-Cartan equation reads
\[d\alpha+\frac{1}{2}[\alpha\wedge\alpha]=0.\]
Decomposing with respect to the Cartan decomposition, we obtain
$ \alpha=\alpha_0+\alpha_1 $ with $\alpha_0\in \Gamma(\mathfrak{k}\otimes T^*M), \
\alpha_1\in \Gamma(\mathfrak{p}\otimes T^*M)$.
And the Maurer-Cartan equation becomes
\[
\left\{\begin{array}{ll}
 & d \alpha_0+\frac{1}{2}[\alpha_0\wedge\alpha_0]+\frac{1}{2}[\alpha_1\wedge\alpha_1]=0.\\
& d
 \alpha_1+[\alpha_0\wedge\alpha_1]=0.
\end{array}\right.
\]
Decompose $\alpha_1$ further into the $(1,0)-$part $\alpha_{1}'$ and the $(0,1)-$part $\alpha_{1}''$. Introducing $\lambda\in S^1$, set
\begin{equation}\label{eq-alphaloop}
   \alpha_{\lambda}=\lambda^{-1}\alpha_{1}'+\alpha_0+\lambda\alpha_{1}'',  \  \lambda\in  S^1.
\end{equation}
It is well known (\cite{DPW}) that the map  $f:M\rightarrow G/K$ is harmonic  if and only if
\[
  d  \alpha_{\lambda}+\frac{1}{2}[\alpha_{\lambda}\wedge\alpha_{\lambda}]=0\ \ \hbox{for all}\ \lambda \in  S^1.
\]

\begin{definition}Let $F(z,\lambda)$ be a solution to the equation $ d  F(z,\lambda)= F(z, \lambda)\alpha_{\lambda},\ F(0,\lambda)=F(0)
.$ Then $F(z,\lambda)$ is called the {\em extended frame} of the harmonic map $f$. Note that $F(z,1)=F(z)$.
 \end{definition}

 \subsubsection{Two decomposition theorems}
 To state the DPW constructions for harmonic maps, we need the Iwasawa and Birkhoff decompositions for loop groups.
 For simplicity, from now on we consider the concrete case for Willmore surfaces \cite{DoWa1}. In this case $G=SO^+(1,n+3)$, $K=SO^+(1,3)\times SO(n)$, and
$\mathfrak{g}=\mathfrak{s}o(1,n+3)=\{X\in \mathfrak{g}l(n+4,\mathbb{R})|X^tI_{1,n+3}+I_{1,n+3}X=0\}.
$
The involution is given by
 \[
\begin{array}{ll}
\sigma: \ SO^+(1,n+3) & \rightarrow \ SO^+(1,n+3)\\
 \ \ \ \ \ ~~\ \ A&  \mapsto \ DAD^{-1},
\end{array}
\quad \quad \hbox{ with} \quad
D=\left(
         \begin{array}{ccccc}
             -I_{4} & 0 \\
            0 & I_{n} \\
         \end{array}
       \right).
\]
Note that $SO^+(1,n+3)^{\sigma}\supset SO^+(1,3)\times SO(n+2)= (SO^+(1,n+3)^{\sigma})_0$. We also have
$
\mathfrak{g}=\mathfrak{k}\oplus\mathfrak{p},
$
with
\[
\mathfrak{k}=\left\{\left(
                   \begin{array}{cc}
                     A_1 &0 \\
                     0 & A_2 \\
                   \end{array}
                 \right)
 |A_1^tI_{1,3}+I_{1,3}A_1 =0, A_2+A_2^t=0\right\}, \  \mathfrak{p}=\left\{\left(
                   \begin{array}{cc}
                   0 & B_1 \\
                     -B_1^tI_{1,3} & 0 \\
                   \end{array}
                 \right)
\right\}.
\]
Let $G^{\mathbb{C}}=SO^+(1,n+3,\mathbb{C}):=  \{X \in SL(n+4,\C) ~|~ X^t I_{1,n+3} X=I_{1,n+3}\}$ with $\mathfrak{so}(1,n+3,\mathbb{C})$ its Lie algebra. Extend $\sigma$ to an inner involution of $SO^+(1,n+3,\mathbb{C})$  with $K^{\mathbb{C}}=S(O^+(1,3,\mathbb{C})\times O(n,\C))$ its fixed point group.

Let $\Lambda G^{\mathbb{C}}_{\sigma}$ be the group of loops in $G^C =SO^+(1,n+3,\mathbb{C})$ twisted by $\sigma$. Then we have:
\begin{theorem}\label{thm-iwasawa} Theorem 4.5 of \cite{DoWa1}, also see \cite{DPW} (Iwasawa decomposition):
There exists a closed, connected solvable subgroup $S \subseteq K^\C$ such that
the multiplication $\Lambda G_{\sigma}^0 \times \Lambda^{+}_S G^{\mathbb{C}}_{\sigma}\rightarrow
\Lambda G^{\mathbb{C}}_{\sigma}$ is a real analytic diffeomorphism onto the open subset
$ \Lambda G_{\sigma}^0 \cdot \Lambda^{+}_S G^{\mathbb{C}}_{\sigma}      = \mathcal{I}^{\mathcal{U}}_e \subset(\Lambda G^{\mathbb{C}}_{\sigma})^0 .
$
Here $\Lambda_{S}^+ G^{\mathbb{C}}_{\sigma}:=\{\gamma\in\Lambda^+G^{S}_{\sigma}~|~\gamma|_{\lambda=0}\in S \}.$
\end{theorem}

Let $\Lambda^-_*G^{\mathbb{C}}_{\sigma}$ be the group of loops that extend holomorphically into $\infty$ and take values $I$ at infinity. Set also
$\Lambda_{\mathcal C}^+ G^{\mathbb{C}}_{\sigma}:=\{\gamma\in\Lambda^+G^{\mathbb{C}}_{\sigma}~|~\gamma|_{\lambda=0}\in (K^\C)^0 \}.
$
\begin{theorem} \label{thm-birkhoff}  \cite{DPW}, \cite{DoWa1}   (Birkhoff decomposition):
The multiplication $\Lambda_{*}^{-} {G}^{\mathbb{C}}_{\sigma}\times
\Lambda^{+}_{\mathcal{C}} {G}^{\mathbb{C}}_{\sigma}\rightarrow
\Lambda {G}^{\mathbb{C}}_{\sigma}$ is an analytic  diffeomorphism onto the
open, dense subset $\Lambda_{*}^{-} {G}^{\mathbb{C}}_{\sigma}\cdot
\Lambda^{+}_{\mathcal{C}} {G}^{\mathbb{C}}_{\sigma}$ {\em (big Birkhoff cell)}.
\end{theorem}

\subsubsection{The DPW construction}
 Let $\mathbb{D}\subset\mathbb{C}$ be a disk or $\mathbb{C}$  with complex coordinate $z$.
\begin{theorem} \label{thm-DPW} \cite{DPW}

 (1) Let $f:\mathbb{D}\rightarrow G/K$ denote a harmonic map with an extended frame $F(z,\bar{z},\lambda)\in \Lambda G_{\sigma}$ and $F(0,0,\lambda)=I$. Then there exists a Birkhoff decomposition of $F(z,\bar{z},\lambda)$
\[
F_-(z,\lambda)=F(z,\bar{z},\lambda)  F_+(z,\bar{z},\lambda),~ \hbox{ with }~ F_+\in\Lambda^+_{S}G^{\mathbb{C}}_{\sigma},
\]
such that $F_-(z,\lambda):\mathbb{D} \rightarrow\Lambda^-_*G^{\mathbb{C}}_{\sigma}$ is meromorphic, with its the Maurer-Cartan form of $F_-$ being the form\[
\eta=F_-^{-1} d  F_-=\lambda^{-1} \eta_{-1}(z) d z,
\]
 with  $\eta_{-1} :\mathbb{D} \rightarrow \mathfrak{p}\otimes\C$ independent of $\lambda$.
The meromorphic $1$-form $\eta$ is called the normalized potential of $f$.

(2). Let $\eta$ be a $\lambda^{-1}\cdot\mathfrak{p}\otimes\C-$valued meromorphic 1-form on $\mathbb{D}$. Let $F_-(z,\lambda)$ be a solution to $F_-^{-1} d  F_-=\eta$, $F_-(0,\lambda)=I$. Then there exists an Iwasawa decomposition
\[
F_-(0,\lambda)=\tilde{F}(z, \bar{z},\lambda)  \tilde{F}^+(z, \bar{z},\lambda),\ \hbox{ with }\ \tilde{F}\in\Lambda G_{\sigma},\ \tilde{F}\in\Lambda ^+_{S} G^{\mathbb{C}}_{\sigma}
\]
on an open subset $\mathbb{D}_{\mathfrak{I}}$ of $\mathbb{D}$. Moreover,  $\tilde{F}(z,\bar{z},\lambda)$ is an extended frame of some harmonic map from $\mathbb{D}_{\mathfrak{I}}$  to $G/K$ with $\tilde{F}(0,\lambda)=I$. All harmonic maps can be obtained in this way, since the above two procedures are inverse to each other if the normalization at some based point is fixed.
\end{theorem}

The normalized potential can be determined from the Maurer-Cartan form of $f$. Let $f$, $F(z,\lambda)$ and $\alpha_{\lambda}=F(z,\lambda)^{-1} d  F(z,\lambda)$ as above. Let $\delta_1$ and $\delta_0$ denote the sum of the holomorphic terms of $z$ about $z=0$ in the Taylor expansion of $\alpha_1^\prime (\frac{\partial}{\partial z})$ and  $\alpha_0^\prime (\frac{\partial}{\partial z})$ respectively.
 \begin{theorem} \label{thm-wu} \cite{Wu} (Wu's formula) We retain the notions in Theorem \ref{thm-DPW}. The the normalized potential of $f$ with respect to the base point $0$ is given by
 \[
 \eta=\lambda^{-1}F_0(z)\delta_1F_0(z)^{-1}  d  z,
\]
where $F_0(z):\mathbb{D}\rightarrow G^{\mathbb{C}}$ is the solution to $F_0(z)^{-1} d  F_0(z)=\delta_0  d  z$, $F_0(0)=I$.
\end{theorem}

\subsubsection{Normalized potentials of totally isotropic Willmore two spheres in $S^6$ }
Let $F$ be a frame of  a Willmore surface $y$ with $\alpha=F^{-1}dF=\alpha_1'+\alpha_0+\alpha_{1}''$ as above. Here
\[ \alpha_0'= \left(
                      \begin{array}{cc}
                        A_1 & 0\\
                       0 & A_2 \\
                      \end{array}
                    \right)dz \ \hbox{ and } \ \alpha_1'= \left(
                      \begin{array}{cc}
                        0 & B_1 \\
                        -B_1^tI_{1,3} & 0 \\
                      \end{array}
                    \right)dz.\]
Let $\delta_1'$ be the holomorphic part of $\alpha_1'$
                    and $\delta_0'$ be the holomorphic part of
 $\alpha_0'$. Let  $\tilde{B}_1$ be the holomorphic part of $B_1$. Let
  \[F_0= \left(
                      \begin{array}{cc}
                        K_1 & 0\\
                       0 & K_2 \\
                      \end{array}
                    \right).\]
                    be the solution to $~F_0^{-1}dF_0=\delta_0',\ F_0(z_0)=I_{8}$.
By Theorem \ref{thm-wu}, we have
  \begin{equation}\label{eq-wu}\eta=F_0\delta_1'F_0^{-1}=\lambda^{-1}\left(
                      \begin{array}{cc}
                        0 & \hat{B}_1 \\
                        -\hat{B}_1^tI_{1,3} & 0 \\
                      \end{array}
                    \right)dz,\ \
  \hbox{ with }\
\hat{B}_1 =K_1 \tilde{B}_1K_2^{-1}.\end{equation}
Applying Wu's formula, we obtain
\begin{theorem}\label{thm-iso-willmore-2}
Let $y$ be a totally isotropic Willmore two sphere in $S^6$. Then the normal bundle of $y$ satisfies the properties  \eqref{eq-normal-bundle} of Theorem \ref{thm-iso-willmore-1}. The normalized potential of $y$ is of the form
\begin{equation}\label{eq-iso-np}
\eta=\lambda^{-1}\eta_{-1}dz=\lambda^{-1}\left(
                      \begin{array}{cc}
                        0 & \hat{B}_1 \\
                        -\hat{B}_1^tI_{1,3} & 0 \\
                      \end{array}
                    \right)dz,\  ~~\hbox{ with }~~  \hat{B}_1=
                                        \left(
                                          \begin{array}{ccccc}
                                            h_{11} & ih_{11}&  h_{12} & ih_{12} \\
                                            h_{21} & ih_{21}&  h_{22} & ih_{22} \\
                                            h_{31} & ih_{31}&  h_{32} & ih_{32} \\
                                            h_{41} & ih_{41}&  h_{42} & ih_{42} \\
                                          \end{array}
                                        \right).\end{equation}
                    Here $h_{ij}$ are meromorphic functions.
\end{theorem}

To derive Theorem \ref{thm-iso-willmore-2}, we first prove the following technical lemma.
\begin{lemma}\label{lemma-ode} Set
\begin{equation}\label{eq-ode}
\mathfrak{k_2^{\C}}:=\left\{A_2\ \left|\
A_2=\left(
          \begin{array}{cccc}
            0 & -\mathrm{b}_{12} &  -\mathrm{b}_{13}  &  -\mathrm{b}_{14}  \\
            \mathrm{b}_{12} & 0 &  \mathrm{b}_{14}  &  -\mathrm{b}_{13}  \\
            \mathrm{b}_{13} &  -\mathrm{b}_{14}  & 0 &  -\mathrm{b}_{34}  \\
            \mathrm{b}_{14} &  \mathrm{b}_{13}  & \mathrm{b}_{34}  & 0 \\
          \end{array}
        \right)\in \mathfrak{so}(4,\C)\right.
\right\} .
\end{equation}
Then, $\mathfrak{k_2^{\C}}$ is a Lie sub-algebra of $ \mathfrak{so}(4,\C)$.  Moreover, let $\mathfrak{K_2^{\C}}$ be the subgroup of $SO(4,\C)$ with Lie algebra $\mathfrak{k_2^{\C}}$. Then
\begin{equation}\label{eq-ode-2}
\mathfrak{K_2^{\C}}={\small \left\{K_2 \left| K_2= \left(
          \begin{array}{cccc}
             \mathrm{t}_{11} & -\mathrm{t}_{12} &  -\mathrm{t}_{13}  &  -\mathrm{t}_{14}  \\
            \mathrm{t}_{12} & \mathrm{t}_{11} &  \mathrm{t}_{14}  &  -\mathrm{t}_{13}  \\
            \mathrm{t}_{13} &  -\mathrm{t}_{14}  &\mathrm{t}_{11} &   \mathrm{t}_{12}  \\
            \mathrm{t}_{14} &  \mathrm{t}_{13}  & -\mathrm{t}_{12}  & \mathrm{t}_{11} \\
          \end{array}
        \right)\left(
                          \begin{array}{cccc}
                            1 &   &   &  \\
                              & 1 &  &   \\
                              &  &\cos \varphi & \sin\varphi \\
                              &  & -\sin\varphi &\cos \varphi \\
                          \end{array}
                        \right)\right.\in SO(4,\C)
\right\}} .
\end{equation}
                        \end{lemma}
\begin{proof}
Let
\[A_2=\left(
          \begin{array}{cccc}
            0 & -\mathrm{b}_{12} &  -\mathrm{b}_{13}  &  -\mathrm{b}_{14}  \\
            \mathrm{b}_{12} & 0 &  \mathrm{b}_{14}  &  -\mathrm{b}_{13}  \\
            \mathrm{b}_{13} &  -\mathrm{b}_{14}  & 0 &  -\mathrm{b}_{34}  \\
            \mathrm{b}_{14} &  \mathrm{b}_{13}  & \mathrm{b}_{34}  & 0 \\
          \end{array}
        \right),\  \widetilde{A_2}=\left(
          \begin{array}{cccc}
            0 & - \widetilde{\mathrm{b}}_{12} &  -\widetilde{\mathrm{b}}_{13}  &  -\widetilde{\mathrm{b}}_{14}  \\
            \widetilde{\mathrm{b}}_{12} & 0 &  \widetilde{\mathrm{b}}_{14}  &  -\widetilde{\mathrm{b}}_{13}  \\
            \widetilde{\mathrm{b}}_{13} &  -\widetilde{\mathrm{b}}_{14}  & 0 &  -\widetilde{\mathrm{b}}_{34}  \\
            \widetilde{\mathrm{b}}_{14} &  \widetilde{\mathrm{b}}_{13}  & \widetilde{\mathrm{b}}_{34}  & 0 \\
          \end{array}
        \right) .\]
        A direct computation shows
        \[[A_2,\tilde{A}_2]=\left(
          \begin{array}{cccc}
            0 &   2\mathrm{b}_{13}\widetilde{\mathrm{b}}_{14}-2\widetilde{\mathrm{b}}_{13}\mathrm{b}_{14} &  -\widehat{\mathrm{b}}_{13}  &  -\widehat{\mathrm{b}}_{14}  \\
          2\widetilde{\mathrm{b}}_{13}\mathrm{b}_{14}- 2\mathrm{b}_{13}\widetilde{\mathrm{b}}_{14} & 0 &  \widehat{\mathrm{b}}_{14}  &  -\widehat{\mathrm{b}}_{13}  \\
            \widehat{\mathrm{b}}_{13} &  -\widehat{\mathrm{b}}_{14}  & 0 &  2\widetilde{\mathrm{b}}_{13}\mathrm{b}_{14}- 2\mathrm{b}_{13}\widetilde{\mathrm{b}}_{14} \\
            \widehat{\mathrm{b}}_{14} &  \widehat{\mathrm{b}}_{13}  &  2\mathrm{b}_{13}\widetilde{\mathrm{b}}_{14}-2\widetilde{\mathrm{b}}_{13}\mathrm{b}_{14}   & 0 \\
          \end{array}
        \right),\]
with
$
\widehat{\mathrm{b}}_{13} =(\mathrm{b}_{12}-\mathrm{b}_{34})\widetilde{\mathrm{b}}_{14}- (\widetilde{\mathrm{b}}_{12}-\widetilde{\mathrm{b}}_{34})\mathrm{b}_{14},\
\widehat{\mathrm{b}}_{14} =(\mathrm{b}_{34}-\mathrm{b}_{12})\widetilde{\mathrm{b}}_{13}- (\widetilde{\mathrm{b}}_{34}-\widetilde{\mathrm{b}}_{12})\mathrm{b}_{13}.\
$
The Lie algebra  $\widetilde{\mathfrak{k_2^{\C}}}$ of
\[ \left\{K_2 \left| K_2= \left(
          \begin{array}{cccc}
             \mathrm{t}_{11} & -\mathrm{t}_{12} &  -\mathrm{t}_{13}  &  -\mathrm{t}_{14}  \\
            \mathrm{t}_{12} & \mathrm{t}_{11} &  \mathrm{t}_{14}  &  -\mathrm{t}_{13}  \\
            \mathrm{t}_{13} &  -\mathrm{t}_{14}  &\mathrm{t}_{11} &   \mathrm{t}_{12}  \\
            \mathrm{t}_{14} &  \mathrm{t}_{13}  & -\mathrm{t}_{12}  & \mathrm{t}_{11} \\
          \end{array}
        \right)\right.\in SO(4,\C)
\right\}\]
is
\[\widetilde{\mathfrak{k_2^{\C}}}:=\left\{A_2\ \left|\
A_2=\left(
          \begin{array}{cccc}
            0 & -\mathrm{b}_{12} &  -\mathrm{b}_{13}  &  -\mathrm{b}_{14}  \\
            \mathrm{b}_{12} & 0 &  \mathrm{b}_{14}  &  -\mathrm{b}_{13}  \\
            \mathrm{b}_{13} &  -\mathrm{b}_{14}  & 0 &  \mathrm{b}_{12}  \\
            \mathrm{b}_{14} &  \mathrm{b}_{13}  & -\mathrm{b}_{12}  & 0 \\
          \end{array}
        \right)\in \mathfrak{so}(4,\C)\right.
\right\} .\]
Since $K_{\varphi}^{-1}\widetilde{\mathfrak{k_2^{\C}}}K_{\varphi}\subset\mathfrak{k_2^{\C}}$, we see that $\mathfrak{K_2^{\C}}$ is the subgroup of $SO(4,\C)$ with Lie algebra $\mathfrak{k_2^{\C}}$. Here
\[K_{\varphi}=\left(
                          \begin{array}{cccc}
                            1 &   &   &  \\
                              & 1 &  &   \\
                              &  &\cos \varphi & \sin\varphi \\
                              &  & -\sin\varphi &\cos \varphi \\
                          \end{array}
                        \right),\]

\end{proof}
\begin{remark} \eqref{eq-ode-2} shows that the real subgroup $\widehat{\mathfrak{K}_2}=\{K_2\in\mathfrak{K}_2| K_2=\bar{K}_2 \}$ is diffeomorphic to $S^3\times S^1$.
\end{remark}\ \\
{\em Proof of Theorem \ref{thm-iso-willmore-2}.}
If $y$ is not S-Willmore, \eqref{eq-normal-bundle} comes from Theorem \ref{thm-iso-willmore-1}. If $y$ is S-Willmore, then let
$E_1$ be a basis of the bundle spanned by $\kappa$ (this bundle is globally defined, see \cite{Ejiri1988} or \cite{MWW} for a proof). Let $E_1,$ $E_2$ be a basis of the bundle spanned by $\kappa,D_z\kappa$ (this bundle is globally defined, see \cite{MWW} for a proof). It is straightforward to verify that \eqref{eq-normal-bundle} holds.

  Now we apply \eqref{eq-normal-bundle}.  Set $E_1=\psi_1+i\psi_2,$ $E_2=\psi_3+i\psi_4$. Then we have a frame $F$ of the form \eqref{F}. Under this framing, we have
  \[B_1=
                                        \left(
                                          \begin{array}{ccccc}
                                            \sqrt{2}\beta_1 & \sqrt{2}i\beta_1&  \sqrt{2}\beta_3 & \sqrt{2}i\beta_3\\
                                            -\sqrt{2}\beta_1 & -\sqrt{2}i\beta_1&  -\sqrt{2}\beta_3 & -\sqrt{2}i\beta_3\\
                                            -k_{1} & -ik_{1}& -k_{3} & -ik_{3} \\
                                            -ik_{1} & k_{1}& -ik_{3} & k_{3}\\
                                          \end{array}
                                        \right) ~ \hbox{ and } A_2=\left(
          \begin{array}{cccc}
            0 & -\mathrm{b}_{12} &  -\mathrm{b}_{13}  &  -\mathrm{b}_{14}  \\
            \mathrm{b}_{12} & 0 &  \mathrm{b}_{14}  &  -\mathrm{b}_{13}  \\
            \mathrm{b}_{13} &  -\mathrm{b}_{14}  & 0 &  -\mathrm{b}_{34}  \\
            \mathrm{b}_{14} &  \mathrm{b}_{13}  & \mathrm{b}_{34}  & 0 \\
          \end{array}
        \right).\]
        Then the normalized potential of $y$ is expressed by \eqref{eq-wu}.

   The holomorphic part  $\tilde{B}_1$ of $B_1$ has the same form as $B_1$ and since $K_1$ does not change the relations between the columns of $\tilde{B}_1$, we need only to consider the influence of $K_2$ on $\tilde{B}_1$.
   Note that $A_2$ takes value in $\mathfrak{k_2^{\C}}$. So the holomorphic part $\tilde{A}_2$ of $A_2$ also takes value in $\mathfrak{k_2^{\C}}$. Therefore, the integration $\hat{A}_2=\int_{z_0}^z\tilde{A}_2dz$ of $\tilde{A}_2$ also takes value in $\mathfrak{k_2^{\C}}$. By Lemma \ref{lemma-ode}, $K_2$ takes value in $\mathfrak{K_2^{\C}}$.
  Summing up, we can assume that the following two equations hold
  \[   K_2= \left(
          \begin{array}{cccc}
             \mathrm{t}_{11} & -\mathrm{t}_{12} &  -\mathrm{t}_{13}  &  -\mathrm{t}_{14}  \\
            \mathrm{t}_{12} & \mathrm{t}_{11} &  \mathrm{t}_{14}  &  -\mathrm{t}_{13}  \\
            \mathrm{t}_{13} &  -\mathrm{t}_{14}  &\mathrm{t}_{11} &   \mathrm{t}_{12}  \\
            \mathrm{t}_{14} &  \mathrm{t}_{13}  & -\mathrm{t}_{12}  & \mathrm{t}_{11} \\
          \end{array}
        \right)\left(
                          \begin{array}{cccc}
                            1 &   &   &  \\
                              & 1 &  &   \\
                              &  &\cos \varphi & \sin\varphi \\
                              &  & -\sin\varphi &\cos \varphi \\
                          \end{array}
                        \right)\]
  and
\[K_1\tilde{B}_1=   \left(
                                          \begin{array}{ccccc}
                                            \widehat{\mathrm{h}}_{11} & i \widehat{\mathrm{h}}_{11}&   \widehat{\mathrm{h}}_{12} & i \widehat{\mathrm{h}}_{12} \\
                                             \widehat{\mathrm{h}}_{21} & i \widehat{\mathrm{h}}_{21}&   \widehat{\mathrm{h}}_{22} & i \widehat{\mathrm{h}}_{22} \\
                                             \widehat{\mathrm{h}}_{31} & i \widehat{\mathrm{h}}_{31}&   \widehat{\mathrm{h}}_{32} & i \widehat{\mathrm{h}}_{32} \\
                                             \widehat{\mathrm{h}}_{41} & i \widehat{\mathrm{h}}_{41}&   \widehat{\mathrm{h}}_{42} & i \widehat{\mathrm{h}}_{42} \\
                                          \end{array}
                                        \right).\]
Then $K_1\tilde{B}_1 K_2^{-1}$ has the form
\[ \left(
                                          \begin{array}{ccccc}
                                            h_{11} & ih_{11}&  h_{12} & ih_{12} \\
                                            h_{21} & ih_{21}&  h_{22} & ih_{22} \\
                                            h_{31} & ih_{31}&  h_{32} & ih_{32} \\
                                            h_{41} & ih_{41}&  h_{42} & ih_{42} \\
                                          \end{array}
                                        \right), \hbox{ with } \left\{\begin{split}h_{j1}=&\widehat{\mathrm{h}}_{j1}(\mathrm{t}_{11}-i\mathrm{t}_{12})-\widehat{\mathrm{h}}_{j2}(\mathrm{t}_{13}+i\mathrm{t}_{14}),\\
h_{j2}=&\left(\widehat{\mathrm{h}}_{j1}(\mathrm{t}_{13}-i\mathrm{t}_{14})+\widehat{\mathrm{h}}_{j2}(\mathrm{t}_{11}+i\mathrm{t}_{12})\right)\\
&\cdot(\cos\varphi-i\sin\varphi),\ 1\leq j\leq4.
\end{split}\right.
\]
  \hfill$\Box $

\begin{remark}Different from the case in $S^4$, where totally isotropic surfaces are automatically S-Willmore surfaces of finite uniton type, totally isotropic surfaces in $S^6$ are not even Willmore in general. Moreover, for a totally isotropic Willmore surface in $S^6$, if the holomorphic $4-$form $\Omega dz^4\neq0$ (hence this surface is not S-Willmore), it is full in $S^6$ and is not of finite uniton type. Given the fact that such surfaces come from the twistor projection of holomorphic or anti-holomorphic curves of the twistor bundle $\mathfrak{T}S^6$ of $S^6$, they can be expressed by rational functions on the Riemann surface.
  The existence of such harmonic maps which are not of finite uniton type is somewhat unexpected since they correspond to holomorphic or anti-holomorphic curves in the twistor bundle of $S^6$.   And it will be an interesting topic to classify and/or to characterize such harmonic maps
  as well as the corresponding Willmore surfaces, especially when the Riemann surface is a torus. As a consequence, it will be an interesting topic to generalize the work of Bohle on Willmore tori \cite{Bohle} to Willmore tori in $S^6$.
\end{remark}

The proof of Theorem \ref{thm-iso-willmore-1} below also reveals other aspects of isotropic surfaces of finite uniton type in $S^6$.\vspace{2mm}
\\
{\em Proof of Theorem \ref{thm-iso-willmore-1}.}
 If $y$ is non S-Willmore with $\langle\kappa,\kappa\rangle=0$, we claim that its normalized potential can only take the form of type (3). By Theorem 3.1 of Section 3, $y$ is totally isotropic and its normal connection has the desired form.

Now let's prove the claim. Theorem 2.8 of \cite{Wang-1} and Theorem 5.2 of \cite{DoWa1} show that $B_1$ must be either of type 2 or of type 3 of Theorem 2.8.

On the other hand, as we have seen before, the isotropy condition and the Willmore equation show
\[\langle\kappa,\kappa\rangle=\langle D_{\bar z}\kappa,\kappa\rangle=\langle D_{\bar z}\kappa,D_{\bar{z}}\kappa\rangle=0.\]
This yields that the Maurer-Cartan form of $y$ satisfies
$B_1 B_1^t=0.$
Then $\tilde{B}_1$, the holomorphic part  of $B_1$,  also satisfies $\tilde{B}_1\tilde{B}_1^t=0.$
As a consequence, we have
\[\hat{B}_1\hat{B}_1^t=K_1 \tilde{B}_1K_2^{-1}(K_2^{-1})^t \tilde{B}_1^tK_1^{t}=K_1 \tilde{B}_1 \tilde{B}_1^tK_1^{t}=0.\]
If the normalized potential $\eta$ of $y$ is of type 2 in Theorem 2.8 of \cite{Wang-1}, then
\[ \hat{B}_{1}=\left(
\begin{array}{cccc}
                                            h_{11} & i {h}_{11} &  h_{12} & f_1{h}_{12}  \\
                                             h_{21} & i {h}_{21} &  h_{12}& f_1{h}_{12}\\
                                            h_{31}& i {h}_{31} &  h_{32}& f_1{h}_{32} \\
                                            h_{41}& i {h}_{41} &  ih_{32}& if_1h_{32} \\
 \end{array}
\right).\]
So the condition $\hat{B}_1\hat{B}_1^t=0$ force $1+f_1^2=0$. So $f_1=i$ or $f_1=-i$. This indicates that $\eta$ is of type 3 (up to a conjugation).
 \hfill$\Box $\\

\section{Construction of totally isotropic Willmore two-spheres in $S^6$}

This section is devoted to an investigation of the geometric properties of Willmore surfaces of type 3 of Theorem 3.3 of \cite{Wang-1}.  For this purpose, by a concrete Iwasawa decomposition, we first provide an algorithm to derive a concrete construction of such Willmore surfaces in $S^6$ from the normalized potentials of  type 3  of Theorem 3.3 of \cite{Wang-1}.  The geometric properties of this kind of Willmore surfaces are also revealed naturally. During this procedure, we will see that Willmore surfaces of this type will be a special kind of totally isotropic Willmore surfaces in $S^6$, which has been shown that their normalized potentials are of type 3 in Section 2.

To achieve these results, we will first transform  $SO^+(1,7,\C)$ isometrically into the Lie group $G(8,\mathbb{C})$, so that the images of normalized potentials of  type 3 in $\mathfrak{g}(8,\C)$ are strictly upper-triangular. This will significantly simplify the computations of the corresponding Iwasawa decompositions. Actually, after this transformation we will be able to derive the corresponding Iwasawa decompositions in a straightforward way.  Moreover, an explicit formula for $y$  is obtained by a detailed discussion of the Maurer-Cartan form derived from the Iwasawa decomposition.
 This presents  an algorithm to construct concrete totally isotropic Willmore two spheres in $S^6$. As illustrations, such new examples  in $S^6$ are constructed this way.

This section has three parts. The main theorem and the new examples are stated first.  The technical lemmas combining a proof of Theorem \ref{th-iso} are stated in the end. The concrete proofs and constructions of examples are postponed to two appendixes.

\subsection{From potentials to surfaces}

We restate the third case of Theorem 3.3 of \cite{Wang-1} by the following theorem

\begin{theorem}\label{th-iso} Let $y$ be a Willmore surface in $S^6$ with its normalized potential being of the form \eqref{eq-iso-np}.
Then $y$ is totally isotropic in $S^6$. Moreover, locally there exists an isotropic frame $\{E_1, E_2\}$ of the normal bundle $V^{\perp}_{\C}$ of $y$ such that \eqref{eq-normal-bundle} holds.
\end{theorem}

\subsection{Examples of totally isotropic Willmore spheres in $S^6$}

We have two kinds examples to illustrate the algorithm presented in the proof of Theorem \ref{th-iso}. The isotropic minimal surfaces in $\R^4$ are used to illustrate the algorithm with simpler computations. The new, totally isotropic, non S-Willmore, Willmore two-sphere in $S^6$ is constructed to answer Ejiri's question explicitly.

\begin{theorem} \label{thm-min-iso}Let \begin{equation} \eta=\lambda^{-1}\left(
                      \begin{array}{cc}
                        0 & \hat{B}_1 \\
                        -\hat{B}_1^tI_{1,3} & 0 \\
                      \end{array}
                    \right)dz,\ \hbox{ with }\  \hat{B}_1=\frac{1}{2}\left(
                     \begin{array}{cccc}
                      -if_2' &  f_2' & 0 & 0 \\
                      if_2'&  -f_2' & 0 & 0 \\
                      f_4' & if_4' & 0 & 0  \\
                      if_4' & -f_4' & 0 & 0  \\
                     \end{array}
                   \right).\end{equation}
Here $f_2$ and $f_4$ are (non-constant) meromorphic functions on  $\C$.
               This $\hat{B}_1$ is of both type (1) and type (3) in Theorem 2.8 of \cite{Wang-1}. The corresponding associated family of Willmore surfaces is
\begin{equation}\label{eq-min-iso}[Y_{\lambda}]=\left[\left(
                                             \begin{array}{cccc}
                       (1+|f_2|^2)-\frac{\bar{f}_2f_4f_2'}{f_4'} -\frac{f_2\bar{f}_4\overline{f_2'}}{\overline{f_4'}}+ \frac{|f_2'|^2(1+|f_4|^2)}{|f_4'|^2}  \\
                        (1-|f_2|^2)+\frac{\bar{f}_2f_4f_2'}{f_4'} +\frac{f_2\bar{f}_4\overline{f_2'}}{\overline{f_4'}}- \frac{|f_2'|^2(1+|f_4|^2)}{|f_4'|^2} \\
                        -\frac{if_2'}{f_4'} +\frac{i\overline{f_2'}}{\overline{f_4'}} \\
                        -\frac{f_2'}{f_4'}-\frac{\overline{f_2'}}{\overline{f_4'}} \\
                     -i(\lambda^{-1}f_2-\lambda\bar{f}_2)+\frac{i\lambda^{-1}f_2'f_4}{f_4'}-\frac{i\lambda\overline{f_2'}\bar{f}_4}{\overline{f_4'}} \\
                     (\lambda^{-1}f_2+\lambda\bar{f}_2)-\frac{\lambda^{-1}f_2'f_4}{f_4'}-\frac{\lambda\overline{f_2'}\bar{f}_4}{\overline{f_4'}} \\
                       0   \\
                       0 \\
                                             \end{array}
                                           \right)\right].\end{equation}
\end{theorem}

\begin{corollary}  The Willmore surface $[Y_{\lambda}]$ in Theorem \ref{thm-min-iso} is conformal to the minimal surface
\begin{equation}\label{eq-min-iso2}x_{\lambda}=\left(
                                             \begin{array}{cccc}
                        -\frac{if_2'}{f_4'} +\frac{i\overline{f_2'}}{\overline{f_4'}} \\
                        -\frac{f_2'}{f_4'}-\frac{\overline{f_2'}}{\overline{f_4'}} \\
                     -i(\lambda^{-1}f_2-\lambda\bar{f}_2)+\frac{i\lambda^{-1}f_2'f_4}{f_4'}-\frac{i\lambda\overline{f_2'}\bar{f}_4}{\overline{f_4'}} \\
                     (\lambda^{-1}f_2+\lambda\bar{f}_2)-\frac{\lambda^{-1}f_2'f_4}{f_4'}-\frac{\lambda\overline{f_2'}\bar{f}_4}{\overline{f_4'}} \\
                                                              \end{array}
                                           \right) \end{equation}
in $\mathbb{R}^4$. Note that our parameter $\lambda$ is different from the usual parameter of the associated family of a minimal surface.
\end{corollary}

\begin{theorem}\label{thm-example}  Let \begin{equation}\label{eq-example-np}\eta=\lambda^{-1}\left(
                      \begin{array}{cc}
                        0 & \hat{B}_1 \\
                        -\hat{B}_1^tI_{1,3} & 0 \\
                      \end{array}
                    \right)dz,\ \hbox{ with } \ \hat{B}_1=\frac{1}{2}\left(
                     \begin{array}{cccc}
                       2iz&  -2z & -i & 1 \\
                       -2iz&  2z & -i & 1 \\
                       -2 & -2i & -z & -iz  \\
                       2i & -2 & -iz & z  \\
                     \end{array}
                   \right).\end{equation}
The associated family of unbranched Willmore two-spheres $x_{\lambda}$, $\lambda\in S^1$, corresponding to $\eta$, is

\begin{equation}\label{example1}
 x_{\lambda} =\frac{1}{ \left(1+r^2+\frac{5r^4}{4}+\frac{4r^6}{9}+\frac{r^8}{36}\right)}
\left(
                          \begin{array}{c}
                            \left(1-r^2-\frac{3r^4}{4}+\frac{4r^6}{9}-\frac{r^8}{36}\right) \\
                            -i\left(z- \bar{z})(1+\frac{r^6}{9})\right) \\
                            \left(z+\bar{z})(1+\frac{r^6}{9})\right) \\
                            -i\left((\lambda^{-1}z^2-\lambda \bar{z}^2)(1-\frac{r^4}{12})\right) \\
                            \left((\lambda^{-1}z^2+\lambda \bar{z}^2)(1-\frac{r^4}{12})\right) \\
                            -i\frac{r^2}{2}(\lambda^{-1}z-\lambda \bar{z})(1+\frac{4r^2}{3}) \\
                            \frac{r^2}{2} (\lambda^{-1}z+\lambda \bar{z})(1+\frac{4r^2}{3})  \\
                          \end{array}
                        \right), \ \hbox{ with $r=|z| .$ }
\end{equation} Moreover $x_{\lambda}:S^2\rightarrow S^6$ is a Willmore immersion in $S^6$, which is full, not S-Willmore, and totally isotropic.
Note that for all $\lambda\in S^1$, $x_{\lambda}$ is isometric to each other in $S^6$.
\end{theorem}
\subsection{Technical lemmas}

\subsubsection{The basic ideas} To begin with, we first explain our basic ideas, since the computations are very technical.
We will divide the proof of Theorem \ref{th-iso} into two steps:

{\em
1.  To derive the harmonic maps from the given normalized potentials;

2. To derive the geometric properties of the corresponding Willmore surfaces.}\\
The main method in Step 1 is a concrete performing of Iwasawa decompositions. The main idea in Step 2 is to read off the Maurer-Cartan forms of the  corresponding Willmore surfaces.

To finish Step 1, we first transform $SO^+(1,7,\C)$ into $G(8,\C)$ (See \eqref{eq-g}) so that the normalized potentials in Theorem \ref{th-iso} are strictly upper-triangular in $\mathfrak{g}(8,\C)=Lie(G(8,\C))$ (Lemma \ref{lemma-iso1}).
Then Lemma \ref{lemma-iso2} provides the concrete expressions of the normalized potential and its meromorphic frame. Lemma \ref{lemma-iso3} gives the  Iwasawa decomposition of the meromorphic frame by the method of undetermined coefficients. This finishes Step 1.

For Step 2, we first derive the forms of the Maurer-Cartan forms of the extended frame derived in Step 1. Then translating into the computations of moving frames, one will obtain the isotropic properties of the corresponding Willmore surfaces.

\subsubsection{Step 1: Iwasawa decompositions}
Set
\begin{equation}\label{eq-g}
G(8,\mathbb{C}):=\{A\in Mat(8,\mathbb{C})| A^tJ_8A=J_{8}, \det A=1
\},
\end{equation}
with $J_{n}=\left( j_{k,l}\right)_{n\times n}, \  j_{k,l}=\delta_{k+l,n+1} \hbox{ for all }1\leq k,l\leq n.$
\begin{lemma}\label{lemma-iso1} Let
\begin{equation}\label{eq-iso-1} \begin{array}{ccc}
                   \check{\mathcal{P}}: SO^+(1,7,\C)&  \rightarrow & G(8,\C)\\
                    A& \mapsto & \check{P}^{-1}\tilde{P}^{-1}A\tilde{P}\check{P},
                 \end{array}\ \end{equation}
with
\[\check{P}=\left(
              \begin{array}{cccc}
                  & J_2  &  &  \\
                J_2  &   &  &  \\
                  &   &  & J_2 \\
                  &   & J_2 &  \\
              \end{array}
            \right),\  \tilde{P}= \frac{1}{\sqrt{2}}\left(
    \begin{array}{cccccccc}
      1 &  &   &  &   &   &   & -1  \\
       1 &  &    &   &   &   & & 1   \\
        & -i &   &  &  &   &  i &   \\
       &  1&   &  &   &   &   1&   \\
       &  & -i  &  &     & i   &   &   \\
       &  & 1  &  &     & 1    &   &   \\
       &  &    & -i & i  &    &   &   \\
       &  &   & 1 &  1 &    &   &   \\
    \end{array}
  \right).\]
Then $\check{\mathcal{P}}$ is a Lie group isomorphism.

We also have that $\check{\mathcal{P}}(SO^+(1,7))=\left\{F\in G(8,\C)\ |\ F=\check{S}_8^{-1}\bar{F}\check{S}_8\right\},$
with \begin{equation}\check{S}_8=\bar{\check{P}}^{-1}\bar{\tilde{P}}^{-1}\tilde{P}\check{P}=
\left(
  \begin{array}{ccc}
    0 & 0 & J_2 \\
    0 & S_4 & 0 \\
    J_2 & 0 & 0 \\
  \end{array}
\right)\ ~ \hbox{ and }\ ~ S_4=\left(
                                                                          \begin{array}{cccc}
                  &   &  & 1 \\
                    & 1   &  &  \\
                  &  &1 &  \\
                              1 &   &  &  \\
                                                                          \end{array}
                                                                        \right).
\end{equation}
This induces an involution of $\Lambda G(8,\C)$
 \begin{equation}\label{eq-def-tau2} \begin{array}{ll}
\check{\tau}:    \Lambda G(8,\mathbb{C})&\rightarrow \Lambda G(8,\mathbb{C})\\
 \ \ \ \ \ \ \ F &\mapsto  \check{S}_8^{-1}\bar{F}\check{S}_8
\end{array}\end{equation}
with
$\check{P}\left(\Lambda SO^+(1,7)\right)=\{F\in\Lambda G(8,\mathbb{C})|\check{\tau}(F)=F\}$
 as its fixed point set.

  The image of the subgroup $(SO^+(1,3 )\times SO( 4))^{\C}$ is
$\check{\mathcal{P}}\left((SO^+(1,3 )\times SO( 4))^{\C}\right) =\{ \check{F}\in G(8,\mathbb{C})\ |\ \check{F}=\check{D}_0^{-1}\check{F}\check{D}_0\}$
with $\check{D}_0=\check{P}^{-1}\tilde{P}^{-1}D\tilde{P}\check{P}=-D_0=\hbox{diag}\left(1,1,-1,-1,-1,-1,1,1\right).$

 Set
\[\check{J}_8=\check{S}_8J_{8}=J_8\check{S}_8=\left(
                                                                          \begin{array}{ccc}
                              I_2    &  &  \\
                  & \check{J}_4 &   \\
                  &   &  I_2 \\
                                                                          \end{array}
                                                                        \right),\ ~ \hbox{ with }\ ~\check{J}_4=S_{4}J_4=\left(
                                                                          \begin{array}{cccc}
                              1&   &  &  \\     &    &  1&  \\
                  &  1& &  \\
                  &   &  & 1 \\
                                                                          \end{array}
                                                                        \right).\]
For any $F\in G(8,\C)$, we have
\begin{equation}\check\tau^{-1}(F)=\check{J}_8\bar{F}^t\check{J}_8.\end{equation}
\end{lemma}

\begin{lemma}\label{lemma-iso2} Let  $\eta$ be the normalized potential of  Theorem \ref{th-iso}. Then
\[\check{\mathcal{P}}(\eta)=\lambda^{-1}\left(
        \begin{array}{ccc}
          0 & \check{f} & 0 \\
          0 & 0 & -\check{f}^{\sharp} \\
          0 & 0 & 0 \\
        \end{array}
      \right)dz,\ ~ \check{f}^{\sharp}:=J_{4}\check{f}^tJ_2,
      \]
with
\begin{equation}\label{eq-B1-to-f}
\check{f}=\left(
            \begin{array}{cccc}
             -h_{32}-i h_{42} &i(h_{12}-h_{22}) &-i (h_{12}+h_{22}) & h_{32}-i h_{42} \\
              -h_{31}-i h_{41} & i (h_{11}- h_{21})& -i(h_{11}+ h_{21}) & h_{31}-i h_{41} \\
            \end{array}
          \right).
\end{equation}
Moreover, $H=I_{8}+\lambda^{-1}H_1+\lambda^{-2}H_2$ is a solution to
\begin{equation}\label{eq-iso-ini}H^{-1}dH=\check{\mathcal{P}}(\eta),\ H|_{z=0}=I_{8}.\end{equation} Here
\[H_1=\left(
        \begin{array}{ccc}
          0 & f & 0 \\
          0 & 0 & -f^{\sharp} \\
          0 & 0 & 0 \\
        \end{array}
      \right),\ H_2=\left(
        \begin{array}{ccc}
          0 & 0 & g \\
          0 & 0 & 0 \\
          0 & 0 & 0 \\
        \end{array}
      \right),\ \hbox{ and } f= \int_0^z\check{f}dz,\ ~~ g=-\int_0^z(f\check{f}^{\sharp})dz.\]
\end{lemma}
 \begin{lemma}\label{lemma-iso3} Retaining the assumptions and the notation of the previous lemmas, assume that $\check{\mathcal{P}}(\eta)$ is the normalized potential of some harmonic map, we obtain:

  Assume that the Iwasawa decomposition of $H$ is
\[H=\check{F}\check{F}_+,\ ~ \hbox{ with }\ ~ \check{F}\in \check{\mathcal{P}}(\Lambda SO^+(1,7)_{\sigma})\subset \Lambda G(8,\C)_{\sigma}\ ~ \hbox{ and }\ ~ \check{F}_+\in\Lambda^+ G(8,\C)_{\sigma}.\]
Then
\begin{equation}\label{eq-check-F}
    \check{F}=H\check{\tau}(W)L_0^{-1}.
\end{equation}
Here $W$, $W_0$ and $L_0$ are the solutions to the matrix equations
\[ \begin{split}
\check{\tau}(H)^{-1}H= WW_0\check{\tau}(W)^{-1}, \ W_0=\check{\tau}(L_0)^{-1}L_0,\\
 \end{split}\]
with \[W=I_8+\lambda^{-1}W_1+\lambda^{-2}W_2,\ \]
      and
      \[ W_1=\left(
        \begin{array}{ccc}
          0 & u & 0 \\
         0 & 0 & -u^{\sharp} \\
          0 &  0 & 0 \\
        \end{array}
      \right),\ W_2=\left(
        \begin{array}{ccc}
         0 & 0 & \check{g} \\
          0 & 0 & 0 \\
          0 & 0 & 0 \\
        \end{array}
      \right), W_0=\left(
        \begin{array}{ccc}
          a & 0 &  0 \\
          0 & q & 0 \\
          0 & 0 & d \\
        \end{array}
      \right),\ L_0=\left(
                   \begin{array}{ccc}
                     l_1 & 0 & 0 \\
                     0 & l_0 & 0 \\
                     0 & 0 & l_4 \\
                   \end{array}
                 \right) .\]
Here the sub-matrices $a$, $q$, $d$ and $u$ are determined by the following equations:
 \begin{subequations}  \label{eq-iso-mc:1}
\begin{align}
&d=I_2+\bar{f}^{t\sharp} \check{J}_4{f}^{ \sharp}+ \bar{g}^{t}g,        \label{eq-iso-mc:1A} \\
& u^{\sharp}d =f^{\sharp}-\check{J}_4\bar{f}^tg,      \label{eq-iso-mc:1B} \\
 &q + u^{\sharp}d\bar{u}^{\sharp t}\check{J}_4 =I_4+\check{J}_4\bar{f}^tf, \label{eq-iso-mc:1C}\\
& a+ uq\check{J}_4\bar{u}^t+ g(\bar{d}^t)^{-1}\bar{g}^t=I_2, \label{eq-iso-mc:1D}\\
& uq-g\bar{u}^{\sharp t}\check{J}_4= f. \label{eq-iso-mc:1E}
\end{align}
\end{subequations}
Moreover,  $\check{F}$ can expressed by these sub-matrices as below
{\small\begin{equation}\label{eq-iso-frame}\check{F}=H\check{\tau}(W)L_0^{-1}=\left(
        \begin{array}{ccc}
         (I-fS_0\bar{u}^{\sharp}J_2+gJ_2 \overline{gd^{-1}}J_2)l_1^{-1} & \lambda^{-1}(f+gJ_2\bar{u}S_0)l_0^{-1} & \lambda^{-2}gl_4^{-1}\\
          -\lambda (S_0\bar{u}^{\sharp}J_2+f^{\sharp}J_2\overline{gd^{-1}}J_2)l_1^{-1}  & (I-f^{\sharp}J_2\bar{u}S_0)l_0^{-1} & -\lambda^{-1}f^{\sharp}l_4^{-1} \\
          \lambda^{2}J_2\overline{gd^{-1}}J_2l_1^{-1}  & \lambda J_2 \bar{u}S_0l_0^{-1} & l_4^{-1} \\
        \end{array}
      \right).
\end{equation}}
 \end{lemma}

 \begin{remark}\
 1.
Since in Lemma \ref{lemma-iso2} the matrices $f$ and $g$, whence also $f^{\sharp}$, are given, equation \eqref{eq-iso-mc:1A} determines $d$, where $d$ is invertible (certainly true fore small $z$ close to $z=0$). Then equation \eqref{eq-iso-mc:1B} determines $u^{\sharp}$, hence $u$. Inserting this into \eqref{eq-iso-mc:1C} results in determining $q$. Inserting what we have so far into \eqref{eq-iso-mc:1D} determines $a$. The last equation, \eqref{eq-iso-mc:1E}, is a consequence of the previous equations. Therefore, the only condition for the solvability of the system of equations is the invertibility of $d$.

 2. If $f$ and $g$ are rational functions of $z$, the invertibility of $d$ is satisfied locally, whence on an open dense subset as a rational expression in $z,\bar{z}$.
\end{remark}

\begin{remark}  1. For a general procedure for the computation of Iwasawa decompositions for algebraic loops, or more generally for rational loops, see $\S I.2$ of \cite{CG}, where an algebraic constructive method is presented to carry out the Birkhoff decomposition for such loop matrices.

  2. In \cite{FSW}, \cite{Co-Pa}, a different method is used to produce all harmonic maps of finite uniton type into $U(n)$, the complex Grassmannian $U(n+m)/(U(n)\times U(m))$, and $G_2$. The treatment of these papers basically follows the spirit of Wood \cite{Wood}, Uhlenbeck \cite{Uh} and Segal \cite{Segal}, using some special unitons. In \cite{S-W}, the converse part of this procedure is also used  for the computation of  the Iwasawa decomposition of elements of the algebraic loop group $\lambda_{alg}U(n)^{\C}$.  Here, we use the DPW methods \cite{DPW} to obtain  harmonic maps of finite uniton type as well as the corresponding Willmore surfaces. The complexified Lie group we use here is in fact isomorphic to $SO(2m,\C)$.  Moreover, due to the complexity of the expressions of Willmore 2-spheres and the corresponding harmonic map, there are in principle no a really very easy way to derive new (not S-Willmore) Willmore 2-spheres, in particular when we not only need to derive the concrete harmonic maps but also need to obtain explicitly the corresponding Willmore surfaces.
\end{remark}

\subsubsection{Step 2: Maurer-Cartan forms}

\begin{lemma}\label{lemma-iso4} Retaining the assumptions and the notation of the previous lemmas, the Maurer-Cartan form  of  $\check{F}$ in \eqref{eq-iso-frame} is of the form
 \begin{equation}\label{eq-iso-m-c1}\check{\alpha}_{\mathfrak{k}}'=\left(
                                \begin{array}{ccc}
                                  a_1 & 0 & 0 \\
                                  0 & a_0 & 0\\
                                  0 & 0 & a_4 \\
                                \end{array}
                              \right)dz,\ ~ \hbox{ and }\ ~\check{\alpha}_{\mathfrak{p}}'=\lambda^{-1}\left(
                                \begin{array}{ccc}
                                  0 & l_1\check{f}l_0^{-1} & 0 \\
                                  0 & 0 & -(l_1\check{f}l_0^{-1})^{\sharp}\\
                                  0 & 0 & 0 \\
                                \end{array}
                              \right)dz,\end{equation}
with
\begin{equation}\label{eq-iso-m-c2}
 \left\{\begin{split}
 a_1&=-l_1\check{f}S_4\bar{u}^{\sharp}J_2l_1^{-1}-l_{1z}l_1^{-1},\\ a_0&=-l_0(\check{f}^{\sharp}J_2\bar{u}S_4-S_4\bar{u}^{\sharp}J_2\check{f})l_0^{-1}-l_{0z}l_0^{-1},\\ a_{4}&=l_4J_2\bar{u}S_4\check{f}^{\sharp}l_4^{-1}-l_{4z}l_4^{-1}.
 \end{split}\right.\end{equation}
 \end{lemma}

Note that these three equations for $a_1$ $a_0$ and $a_4$ actually should  be read as ordinary differential equations for $l_1$, $l_0$ and $l_4$. As initial conditions we may use $l_j(0)=I, j=0,1,4.$

\begin{lemma}\label{lemma-iso5} Let $\mathcal{F}:M\rightarrow SO^+(1,7)/SO^+(1,3)\times SO(4)$ be the conformal Gauss map of a Willmore surface $y$, with an extended frame $F$. If the  Maurer-Cartan form of $\check{F}=\check{\mathcal{P}}(F)$ has the form
\eqref{eq-iso-m-c1}, then $y$ is totally isotropic in $S^{6}$. Moreover, locally there exists an isotropic frame $\{E_1, E_2\}$ of the normal bundle $V^{\perp}_{\C}$ of $y$ such that  \eqref{eq-normal-bundle} holds.\end{lemma}

Lemma \ref{lemma-iso1}, Lemma \ref{lemma-iso2} can be verified by straightforward matrix computations since the concrete formulas are provided (Compare also \cite{Wang-1}). So we leave these computations  to the  readers.  The proofs of the other lemmas will be contained in the following section.

\section{Appendix A: Iwasawa decompositions}
\subsection{Proof of Lemma \ref{lemma-iso3}}
  First  one computes
 \begin{equation*}\begin{split}\check{\tau}^{-1}(H)H&=\left(
        \begin{array}{ccc}
          I_2 & 0 & 0 \\
          \lambda\check{J}_4\bar{f}^t & I_4 & 0 \\
          \lambda^2 \bar{g}^t & -\lambda\bar{f}^{\sharp,t}\check{J}_4 & I_2 \\
        \end{array}\right)\left(
        \begin{array}{ccc}
          I_2 & \lambda^{-1}f & \lambda^{-2}g \\
          0 & I_4 & -\lambda^{-1}f^{\sharp} \\
          0 & 0 & I_2 \\
        \end{array}
      \right)\\
      &=\left(
        \begin{array}{ccc}
          I_2 & \lambda^{-1}f & \lambda^{-2}g \\
        \lambda\check{J}_4\bar{f}^t   & I_4+\check{J}_4\bar{f}^tf & \lambda^{-1}\check{J}_4\bar{f}^tg-\lambda^{-1}f^{\sharp} \\
         \lambda^2 \bar{g}^t & \lambda\bar{g}^t f-\lambda\bar{f}^{\sharp,t}\check{J}_4 & I_2+\lambda\bar{f}^{\sharp,t}\check{J}_4f^{\sharp}+\bar{g}^tg \\
        \end{array}
      \right).\\
      \end{split}
      \end{equation*}
We write $\check\tau^{-1}(H)H=WW_0\check\tau^{-1}(W)$ with
\[ W=I_8+\lambda^{-1}W_1+\lambda^{-2}W_2\]
and
\[ W_1=\left(
        \begin{array}{ccc}
          0 & u & 0 \\
         -v^{\sharp} & 0 & -u^{\sharp} \\
          0 &  v & 0 \\
        \end{array}
      \right),\ W_0=\left(
        \begin{array}{ccc}
          a & 0 & b \\
          0 & q & 0 \\
          c & 0 & d \\
        \end{array}
      \right),\ W_0^{-1}=\left(
        \begin{array}{ccc}
          \hat{a} & 0 & \hat{b} \\
          0 & q^{-1} & 0 \\
          \hat{c} & 0 & \hat{d} \\
        \end{array}
      \right).\]
Hence we obtain
\begin{equation}\label{eq-iso-iwa}\left\{
\begin{split}
&W_2W_0=H_2,\\
&W_1W_0+W_2W_0\check{J}_8\bar{W}_1^t\check{J}_8=H_1+\check{J}_8\bar{H}_1^t\check{J}_8H_2,\\
&W_0+W_1W_0\check{J}_8\bar{W}_1^t\check{J}_8+W_2W_0\check{J}_8\bar{W}_2^t\check{J}_8=I+\check{J}_8\bar{H}_1^t\check{J}_8H_1+\check{J}_8\bar{H}_2^t\check{J}_8H_2.
\end{split}
\right.
\end{equation}
Since
\[\check{J}_8\bar{W}_1^t\check{J}_8=\left(
        \begin{array}{ccc}
          0 & -\bar{v}^{\sharp t}\check{J}_4 & 0 \\
         \check{J}_4\bar{u}^t & 0 &  \check{J}_4\bar{v}^t\\
         0& -\bar{u}^{\sharp t}\check{J}_4 &  0\\
        \end{array}
      \right),\ W_1W_0=\left(
        \begin{array}{ccc}
          0 & uq & 0 \\
       -v^{\sharp}a-u^{\sharp}c& 0 & -v^{\sharp}b-u^{\sharp}d \\
          0 &vq & 0 \\
        \end{array}
      \right),\]and      \[W_2W_0\check{J}_8\bar{W}_1^t\check{J}_8=H_2\check{J}_8\bar{W}_1^t\check{J}_8=\left(
        \begin{array}{ccc}
          0 & -g\bar{u}^{\sharp t}\check{J}_4 & 0 \\
         0& 0 &  0\\
         0& 0 &  0\\
        \end{array}
      \right),\]
from the second matrix equation of \eqref{eq-iso-iwa}, one derives easily that
\[vq=0,\ -v^{\sharp}a-u^{\sharp}c=0,\ uq-g\bar{u}^{\sharp t}\check{J}_4= f,\
 -v^{\sharp}b-u^{\sharp}d =\check{J}_4\bar{f}^tg-f^{\sharp}.\]
Since $q$ is invertible, $v=0$. Therefore we have
\begin{equation*}
v=0,\ u^{\sharp}c=0,\  uq-g\bar{u}^{\sharp t}\check{J}_4= f,\
 u^{\sharp}d =f^{\sharp}-\check{J}_4\bar{f}^tg.
\end{equation*}
Next we consider the third matrix equation in \eqref{eq-iso-iwa}. Since
\begin{equation*}
  W_1W_0\check{J}_8\bar{W}_1^t\check{J}_8 =\left(
        \begin{array}{ccc}
          \cdots& 0  & 0 \\
         0& \cdots & 0  \\
          0& 0 & 0\\
        \end{array}
      \right) ~\hbox{ and }~
       W_2W_0\l\check{J}_8\bar{W}_2^t\check{J}_8 =H_2\check{J}_8\bar{W}_2^t\check{J}_8=\left(
        \begin{array}{ccc}
          \cdots& 0  & 0 \\
         0& 0 & 0  \\
          0& 0 & 0 \\
        \end{array}
      \right),
   \end{equation*}
comparing with the $\lambda-$independent part of $\check\tau^{-1}(H)H$,
we derive directly that $c=b=\hat{c}=\hat{b}=0.$
Substituting these results into the matrix equations in \eqref{eq-iso-iwa}, a straightforward computation yields  \eqref{eq-iso-mc:1}.

In the end, let $L_0$ be of the form as in Lemma \ref{lemma-iso3}, it is easy to compute
\begin{equation*}
\begin{split}\check{F}&=H\check{\tau}(W)L_0^{-1}\\
&=\left(
        \begin{array}{ccc}
          I & \lambda^{-1}f & \lambda^{-2}g \\
          0 & I & -\lambda^{-1}f^{\sharp} \\
          0 & 0 & I \\
        \end{array}
      \right)\left(
        \begin{array}{ccc}
          I &  &  \\
        -\lambda S_4\bar{u}^{\sharp}J_2 & I & \\
          \lambda^2 J_2\bar{gd^{-1}}J_2 & \lambda J_2\bar{u}S_4 & I \\
        \end{array}
      \right)L_0^{-1}\\
      &=\left(
        \begin{array}{ccc}
         (I-fS_4\bar{u}^{\sharp}J_2+gJ_2 \overline{gd^{-1}}J_2)l_1^{-1} & \lambda^{-1}(f+gJ_2\bar{u}S_4)l_0^{-1} & \lambda^{-2}gl_4^{-1}\\
          -\lambda (S_4\bar{u}^{\sharp}J_2+f^{\sharp}J_2\overline{gd^{-1}}J_2)l_1^{-1}  & (I-f^{\sharp}J_2\bar{u}S_4)l_0^{-1} & -\lambda^{-1}f^{\sharp}l_4^{-1} \\
          \lambda^{2}J_2\overline{gd^{-1}}J_2l_1^{-1}  & \lambda J_2 \bar{u}S_4l_0^{-1} & l_4^{-1} \\
        \end{array}
      \right).\end{split}\end{equation*}
\hfill$\Box $

\subsection{The Maurer-Cartan form of $\check{F}$ and the geometry of Willmore surfaces}\ \\
{\em Proof of Lemma \ref{lemma-iso4}:} We have
\[\check{F}^{-1}d\check{F}=\lambda^{-1} \check\alpha_{\mathfrak{p}}'+
\check{\alpha}_{\mathfrak{k}}+\lambda\check{\alpha}_{\mathfrak{p}}'',\] with
$\check\alpha_{\mathfrak{p}}'=L_0\check{\mathcal{P}}\left(\eta_{-1}\right)L_0^{-1}dz,\ \check\alpha_{\mathfrak{k}}'=L_0[\check{\mathcal{P}}\left(\eta_{-1}\right),\check\tau(W_1)]L_0^{-1}dz+L_0(L_0^{-1})_zdz.$
 Since \[\check{\mathcal{P}}\left(\eta_{-1}\right)=\left(
                                \begin{array}{ccc}
                                  0 & \check{f}  & 0 \\
                                  0 & 0 & -\check{f}^{\sharp} \\
                                  0 & 0 & 0 \\
                                \end{array}
                              \right) \ \hbox{ and }\ \check\tau(W_1)=\left(
                         \begin{array}{ccc}
                            0& 0 &0  \\
                           -S_4\bar{u}^{\sharp}J_2 & 0 &0  \\
                           0 & J_2\bar{u}S_4 & 0 \\
                         \end{array}
                       \right),
\]
we obtain
\[L_0[\check{\mathcal{P}}\left(\eta_{-1}\right),\check\tau(W_1)]L_0^{-1}=\left(
                         \begin{array}{ccc}
                           -l_1f'S_4\bar{u}^{\sharp}J_2l_1^{-1} &  &  \\
                           & -l_0(f'^{\sharp}J_2\bar{u}S_4-S_4\bar{u}^{\sharp}Jf')l_0^{-1} &  \\
                            &  & l_4J_2\bar{u}S_4f'^{\sharp}l_4^{-1}  \\
                         \end{array}
                       \right).\]
Now Lemma \ref{lemma-iso4} follows.
\hfill$\Box $\\
\\
{\em Proof of Lemma \ref{lemma-iso5}:}
  By \eqref{eq-iso-m-c1} in  Lemma \ref{lemma-iso4}, there exists a frame $\check{F}$ such that $(1,0)-$part $\check{\alpha}'$  of the Maurer-Cartan form of $\check{F}$ has the form
  \[\left(
   \begin{array}{cccccccc}
     \check{c}_{11} & \check{c}_{12} & \check{b}_{11} & \check{b}_{12}  & \check{b}_{13} & \check{b}_{14}  & 0 & 0 \\
    \check{c}_{21} & \check{c}_{22} & \check{b}_{21} & \check{b}_{22}  & \check{b}_{23} & \check{b}_{24}  & 0 & 0 \\
       0   &    0   & \check{s}_{11} & \check{s}_{12} & \check{s}_{13} & 0 &  -\check{b}_{24} &  -\check{b}_{14} \\
       0   &    0   & \check{s}_{21} & \check{s}_{22} & 0   &  -\check{s}_{13} &  -\check{b}_{23} &  -\check{b}_{13} \\
       0   &    0   & \check{s}_{31} & 0 & \check{s}_{33} & -\check{s}_{12} &  -\check{b}_{22} &  -\check{b}_{12} \\
       0   &    0   & 0& -\check{s}_{31} & -\check{s}_{21} & -\check{s}_{11}  &  -\check{b}_{21} &  -\check{b}_{11} \\
       0   &    0   & 0& 0 & 0 & 0  &  -\check{c}_{22} &  -\check{c}_{12} \\
       0   &    0   & 0& 0 & 0 & 0  &  -\check{c}_{21} &  -\check{c}_{11} \\
   \end{array}
 \right)dz.\]
Set
$F=\mathcal{P}^{-1}(\tilde{F})=(\phi_1,\phi_2,\phi_3,\phi_4,\psi_1,\psi_2,\psi_3,\psi_{4}).$  By \eqref{eq-iso-1} in Lemma \ref{lemma-iso1}, we derive that
\[\alpha'=F^{-1}F_{z}dz=\check{\mathcal{P}}^{-1}(\tilde \alpha')=\left(
                   \begin{array}{cc}
                     A_1 & \lambda^{-1}B_1 \\
                     -\lambda^{-1}B_1^tI_{1,3} & A_2 \\
                   \end{array}
                 \right) d z,\]
                 with
                 \[A_1=\left(
   \begin{array}{cccccccc}
     0 & s_{22} & s_{13} &s_{14} \\
     s_{22} & 0 & s_{23} &  s_{24}\\
     s_{13} & -s_{23} &  0 & -is_{11}  \\
     s_{14} & -s_{24} & is_{11} & 0  \\
    \end{array}
 \right),~~\ ~~ \left\{\begin{split}
 2s_{13}&= -i(\check{s}_{12}-\check{s}_{13})-i(\check{s}_{31}-\check{s}_{21}),\\
 2s_{14}&=(\check{s}_{12}-\check{s}_{13})+(\check{s}_{21}-\check{s}_{31}),\\
 2s_{23}&= i(\check{s}_{12}+\check{s}_{13}+\check{s}_{21}+\check{s}_{31}), \\
 2s_{24}&=(\check{s}_{12}+\check{s}_{13})+(\check{s}_{31}+\check{s}_{21}), \\ \end{split}\right.\]
  \[A_2=\frac{1}{2}\left(
                      \begin{array}{cccc}
                        0 & -2i\check{c}_{22} & \check{c}_{21}-\check{c}_{12} & -i(\check{c}_{12}+\check{c}_{21}) \\
                        2i\check{c}_{11} & 0 & i(\check{c}_{12}+\check{c}_{21}) & \check{c}_{12}-\check{c}_{21} \\
                        \check{c}_{12}-\check{c}_{21} & -i(\check{c}_{12}+\check{c}_{21}) & 0 & -2i\check{c}_{11} \\
                        i(\check{c}_{12}+\check{c}_{21}) & \check{c}_{21}-\check{c}_{12} & 2i\check{c}_{11} &  0  \\
                      \end{array}
                    \right),
      \]
     and
\begin{equation}\begin{split}B_1&=\frac{1}{2}\left(
   \begin{array}{cccc}
    i(\check{b}_{23}-\check{b}_{22}) & -(\check{b}_{23}-\check{b}_{22})   & i(\check{b}_{13}-\check{b}_{12}) & -(\check{b}_{13}-\check{b}_{12})\\
    i(\check{b}_{23}+\check{b}_{22}) & -(\check{b}_{23}+\check{b}_{22})   & i(\check{b}_{13}+\check{b}_{12}) & -(\check{b}_{13}+\check{b}_{12})\\
    \check{b}_{24}-\check{b}_{21}   & i(\check{b}_{24}-\check{b}_{21})   & \check{b}_{14}-\check{b}_{11} & i(\check{b}_{14}-\check{b}_{11})\\
    i(\check{b}_{24}+\check{b}_{21}) & -(\check{b}_{24}+\check{b}_{21})   & i(\check{b}_{14}+\check{b}_{11}) & -(\check{b}_{14}+\check{b}_{11})\\
\end{array}
 \right)\\
 &=\left(
           \begin{array}{cccc}
             h_{11} & ih_{11} & h_{13} & ih_{13} \\
             h_{21} & ih_{21} & h_{23} & ih_{23} \\
             h_{31} & ih_{31} & h_{33} & ih_{33} \\
             h_{41} & ih_{41} & h_{43} & ih_{43} \\
           \end{array}
         \right)
  .\end{split}
\end{equation}
Therefore, one obtains
\begin{equation}\label{eq-iso-derivative1}\left\{\begin{split}\phi_{1z}&=\lambda^{-1}(h_{11}(\psi_1+i\psi_2)+ h_{13}(\psi_3+i\psi_4)) \ \mod \{\phi_{1},\phi_{2},\phi_{3},\phi_{4}\}\\
\phi_{jz}&=-\lambda^{-1}(h_{j1}(\psi_1+i\psi_2)+ h_{j3}(\psi_3+i\psi_4)) \ \mod \{\phi_{1},\phi_{2},\phi_{3},\phi_{4}\}, \ \ j=2,3,4.\end{split}\right.
\end{equation}
Now assume  that $Y$ is a canonical lift of the   Willmore surface $y$. Note that $Span_{\C}\{Y,Y_z,Y_{\bar{z}},N\}=Span_{\C}\{\phi_{1},\phi_{2},\phi_{3},\phi_{4}\}.$
So $Y_z$ is a linear combination of $\{\phi_{1},\phi_{2},\phi_{3},\phi_{4}\}$.
Then we compute the Hopf differential of $Y$:
\begin{equation*}\begin{split}\kappa&=Y_{zz} \mod \{Y,Y_z,Y_{\bar{z}},N\}\\
&= \lambda^{-1}k_1(\psi_1+i\psi_2)+\lambda^{-1}k_3(\psi_3+i\psi_4),\\
\end{split} \end{equation*}
with $k_1,k_3$ some complex-valued functions. Whence $\langle\kappa,\kappa\rangle\equiv0$, i.e., $\kappa$ is isotropic. To show that $Y$ is totally isotropic, we need only to verify that $D_{z}\kappa$ is isotropic. From the Maurer-Cartan form of $F$, we derive that
\begin{equation*}\begin{split} \psi_{1z}&=i\check{c}_{22}\psi_2+\frac{\check{c}_{12}-\check{c}_{21}}{2}\psi_3+\frac{i(\check{c}_{12}+\check{c}_{21})}{2}\psi_4 \ \mod \{\phi_{1},\phi_{2},\phi_{3},\phi_{4}\}, \\\psi_{2z}&=-i\check{c}_{22}\psi_1+\frac{-i(\check{c}_{12}+\check{c}_{21})}{2}\psi_3+\frac{\check{c}_{12}-\check{c}_{21}}{2}\psi_4 \ \mod \{\phi_{1},\phi_{2},\phi_{3},\phi_{4}\}, \\
\psi_{3z}&=\frac{\check{c}_{21}-\check{c}_{12}}{2}\psi_1+\frac{i(\check{c}_{12}+\check{c}_{21})}{2}\psi_2+i\check{c}_{11}\psi_4 \ \mod \{\phi_{1},\phi_{2},\phi_{3},\phi_{4}\}, \\
\psi_{4z}&=\frac{-i(\check{c}_{12}+\check{c}_{21})}{2}\psi_1+\frac{\check{c}_{21}-\check{c}_{12} }{2}\psi_2-i\check{c}_{11}\psi_3 \ \mod \{\phi_{1},\phi_{2},\phi_{3},\phi_{4}\}, \\
\end{split} \end{equation*}
So
\[\left\{\begin{split}(\psi_1+i\psi_2)_{z}&=\check{c}_{22}(\psi_1+i\psi_2)+\check{c}_{12}(\psi_3+i\psi_4) \ \mod \{\phi_{1},\phi_{2},\phi_{3},\phi_{4}\}, \\
(\psi_3+i\psi_4)_{z}&=\check{c}_{21}(\psi_1+i\psi_2)+\check{c}_{11}(\psi_3+i\psi_4) \ \mod \{\phi_{1},\phi_{2},\phi_{3},\phi_{4}\}. \\
\end{split}\right.\]
As a consequence, we obtain that
\[D_z\kappa= \lambda^{-1}\left(\delta_1(\psi_1+i\psi_2)+\delta_3(\psi_3+i\psi_4)\right)\]
for some complex valued function $\delta_1$, $\delta_2$. This indicates that $D_z\kappa$ is also isotropic, i.e. $Y$ as well as $y$ is totally isotropic. \hfill$\Box $

\subsection{An Algorithm to derive Willmore surfaces from frames}

This subsection is to derive an algorithm permitting to read off $y$ from the frame $F$. Although the harmonic maps have been constructed in the above  subsections, to obtain the Willmore surfaces from the harmonic maps needs more computations. We retain the notation in the proof of Lemma \ref{lemma-iso5}.

Set
\[B_1=(h_1,ih_1,h_3,ih_3),~ \ \hbox{ with }\ h_j=(h_{1j},h_{2j},h_{3j},h_{4j})^t, \ j=1,\ 3.\]
Since $B_1$ satisfies $B_1^tI_{1,3}B_1=0$, we have
$h_j^tI_{1,3}h_l=0,\ j, \ l=1,\ 3.$  Therefore $h_1$ and $h_2$ are contained in one of the following two subspaces (see also \cite{Wang-1}):
\[Span_{\C}\left\{\left(
                    \begin{array}{c}
                      1+\rho_1\rho_2 \\
                      1-\rho_1\rho_2 \\
                      \rho_1+\rho_2 \\
                      -i(\rho_1-\rho_2)\\
                    \end{array}
                  \right), \left(
                    \begin{array}{c}
                      \rho_1  \\
                       -\rho_1 \\
                      1 \\
                      i\\
                    \end{array}
                  \right)
\right\},\ \hbox{ or }\ Span_{\C}\left\{\left(
                    \begin{array}{c}
                      1+\rho_1\rho_2 \\
                      1-\rho_1\rho_2 \\
                      \rho_1+\rho_2 \\
                      -i(\rho_1-\rho_2)\\
                    \end{array}
                  \right),\left(
                    \begin{array}{c}
                      \rho_2  \\
                       -\rho_2 \\
                      1 \\
                      -i\\
                    \end{array}
                  \right)
\right\}.\]

Let $Y$ be a canonical lift of $y$. Hence $Y\in Span_{\R}\{\phi_{1},\phi_{2},\phi_{3},\phi_{4}\}$. Since $Y$ is real and lightlike, we may assume that \begin{equation}\label{eq-y}
Y=\hat{\rho}_0\left((1+|\hat{\rho}_1|^2)\phi_1+(1-|\hat{\rho}_1|^2)\phi_2+(\hat{\rho}_1+\bar{\hat{\rho}}_1)\phi_3
-i(\hat{\rho}_1-\bar{\hat{\rho}}_1)\phi_4\right),\end{equation}
with $\hat{\rho}_0\neq0$.
A straightforward computation by use of \eqref{eq-iso-derivative1} yields
\begin{equation*}\begin{split}
Y_z=&
\hat{\rho}_0\left((1+|\hat{\rho}_1|^2)h_{11}-(1-|\hat{\rho}_1|^2)h_{21} +(\hat{\rho}_1+\bar{\hat{\rho}}_1)h_{31}-i(\hat{\rho}_1-\bar{\hat{\rho}}_1)h_{41}\right)(\psi_1+i\psi_2) \\
 & ~ \ +\hat{\rho}_0\left((1+|\hat{\rho}_1|^2)h_{13}-(1-|\hat{\rho}_1|^2)h_{23} +(\hat{\rho}_1+\bar{\hat{\rho}}_1)h_{33}-i(\hat{\rho}_1-\bar{\hat{\rho}}_1)h_{43}\right)(\psi_3+i\psi_4)\\
 &~
  \ \mod \{\phi_{1},\phi_{2},\phi_{3},\phi_{4}\}.\\
 \end{split}\end{equation*}
Hence, to ensure that $Y_z\in Span_{\C}\{\phi_{1},\phi_{2},\phi_{3},\phi_{4}\}$, $\hat{\rho}_1$ needs to satisfy
\begin{equation}\label{eq-iso-derivative2}\left\{\begin{split}
 & (1+|\hat{\rho}_1|^2)h_{11}-(1-|\hat{\rho}_1|^2)h_{21} +(\hat{\rho}_1+\bar{\hat{\rho}}_1)h_{31}-i(\hat{\rho}_1-\bar{\hat{\rho}}_1)h_{41}=0,\\
 & (1+|\hat{\rho}_1|^2)h_{13}-(1-|\hat{\rho}_1|^2)h_{23} +(\hat{\rho}_1+\bar{\hat{\rho}}_1)h_{33}-i(\hat{\rho}_1-\bar{\hat{\rho}}_1)h_{43}=0. \\
 \end{split}\right.
 \end{equation}
 Without loss of generality we assume that
  $$h_1=\rho_0(1+\rho_1\rho_2,1-\rho_1\rho_2,\rho_1+\rho_2,-i(\rho_1-\rho_2))^t.$$

If the maximal rank of $B_1$ is $1$, then $$h_3 =\rho_{01}(1+\rho_1\rho_2,1-\rho_1\rho_2,\rho_1+\rho_2,-i(\rho_1-\rho_2))^t.$$ So \eqref{eq-iso-derivative2} is equivalent to
$(\hat{\rho}_1-\rho_1)( \rho_2-\bar{\hat{\rho}}_1)=0.$ Hence $$\hat{\rho}_1=\rho_1 \ \hbox{ or } \ \hat{\rho}_1=\bar{\rho}_2.$$
These two solutions provide a pair of dual Willmore (therefore S-Willmore) surfaces $y$ and $\hat{y}$ with the same conformal Gauss map.

If the maximal rank of $B_1$ is $2$, then
$$h_3 =\rho_{01}(1+\rho_1\rho_2,1-\rho_1\rho_2,\rho_1+\rho_2,-i(\rho_1-\rho_2))^t+\rho_{02}(\rho_1,-\rho_1,1,i)^t,$$
 or
 $$h_3 =\rho_{01}(1+\rho_1\rho_2,1-\rho_1\rho_2,\rho_1+\rho_2,-i(\rho_1-\rho_2))^t+\rho_{02}(\rho_2,-\rho_2,1,-i)^t.$$
 For the first case, \eqref{eq-iso-derivative2} is equivalent to
$$\hat{\rho}_1=\rho_1.$$
 For the second case, \eqref{eq-iso-derivative2} is equivalent to
$$\hat{\rho}_1=\bar{\rho}_2.$$
In both cases, we obtain a unique non-S-Willmore Willmore surface.

 From the above discussions it is clear that it is necessary to obtain the first four columns of $F$. By \eqref{eq-iso-frame}, $\check{F}$ can be derived from the Iwasawa decomposition. Set $\check{F}=(\mathrm{f}_{jl})$, $j,l=1,\cdots,8$, and $F=\check{\mathcal{P}}^{-1}(\check{F})$. Writing \[F=(\phi_1,\phi_2,\phi_3,\phi_4,\psi_1,\psi_2,\psi_3,\psi_{4}),\]
 and setting
$(\hat\phi_1,\hat\phi_2,\hat\phi_3,\hat\phi_4)=(\phi_1+\phi_2,\phi_1-\phi_2,  \phi_3-i\phi_4,\phi_3+i\phi_4) $
 one obtains straightforwardly from \eqref{eq-iso-1} that
\begin{equation}\label{eq-iso-frame2}
(\hat\phi_1,\hat\phi_2,\hat\phi_3,\hat\phi_4)= \left(
                               \begin{array}{cccc}
                                 ( \mathrm{f}_{44}-\mathrm{f}_{54}) & -(\mathrm{f}_{45}-\mathrm{f}_{55})  & -i(\mathrm{f}_{46}-\mathrm{f}_{56})  & i(\mathrm{f}_{43}-\mathrm{f}_{53})  \\
                                 (\mathrm{f}_{44}+\mathrm{f}_{54}) & -(\mathrm{f}_{45}+\mathrm{f}_{55})  & -i(\mathrm{f}_{46}+\mathrm{f}_{56})  & i(\mathrm{f}_{43}+\mathrm{f}_{53})  \\
                                 -i(\mathrm{f}_{34}-\mathrm{f}_{64}) & i(\mathrm{f}_{35}-\mathrm{f}_{65})  & -(\mathrm{f}_{36}-\mathrm{f}_{66})  & (\mathrm{f}_{33}-\mathrm{f}_{63})   \\
                                 (\mathrm{f}_{34}+\mathrm{f}_{64}) & -(\mathrm{f}_{35}+\mathrm{f}_{65})  &  -i(\mathrm{f}_{36}+\mathrm{f}_{66})  & i(\mathrm{f}_{33}+\mathrm{f}_{63})   \\
                                 -i(\mathrm{f}_{24}-\mathrm{f}_{74}) & i(\mathrm{f}_{25}-\mathrm{f}_{75})  & -(\mathrm{f}_{26}-\mathrm{f}_{76})  & (\mathrm{f}_{23}-\mathrm{f}_{73})   \\
                                 (\mathrm{f}_{24}+\mathrm{f}_{74}) & -(\mathrm{f}_{25}+\mathrm{f}_{75})  & -i(\mathrm{f}_{26}+\mathrm{f}_{76})  & i(\mathrm{f}_{23}+\mathrm{f}_{73})   \\
                                 -i(\mathrm{f}_{14}-\mathrm{f}_{84}) &i(\mathrm{f}_{15}-\mathrm{f}_{85})  &  -(\mathrm{f}_{16}-\mathrm{f}_{86})  & (\mathrm{f}_{13}-\mathrm{f}_{83})  \\
                                 (\mathrm{f}_{14}+\mathrm{f}_{84}) & -(\mathrm{f}_{15}+\mathrm{f}_{85})  & -i(\mathrm{f}_{16}+\mathrm{f}_{86})  & i(\mathrm{f}_{13}+\mathrm{f}_{83})   \\
                               \end{array}
                             \right).\end{equation}

\section{Appendix B: Construction of Examples}

\subsection{Proof of Theorem \ref{thm-min-iso}}\ \hspace{1mm}  From the procedure shown in Section 4.3, to derive the expression of $y$, one needs to figure out $B_1$ of the Maurer-Cartan form and the first four columns of the frame $F$.
Applying Lemma \ref{lemma-iso2} and Lemma \ref{lemma-iso3} to $ \check{\mathcal{P}}(\eta)$, $F$ and the Maurer-Cartan form can be derived by solving equation \eqref{eq-iso-mc:1} for the Iwasawa decomposition. Therefore we have three steps to derive $y$:

  1. Computation of the first four columns of $F$;

  2. Computation of the Maurer-Cartan form of $F$;

  3. Computation of $Y$.\\

{\em Step 1: Computation of the first four columns of $F$.}
By \eqref{eq-B1-to-f}, it is straightforward to derive that \[\check{\mathcal{P}}(\eta)=\lambda^{-1}\left(
                                \begin{array}{ccc}
                                  0 & \check{f}  & 0 \\
                                  0 & 0 & -\check{f}^{\sharp} \\
                                  0 & 0 & 0 \\
                                \end{array}
                              \right)dz\]
                              with
                              \[\check{f}=\left(
                     \begin{array}{cccc}
                       0&  0 & 0 & 0 \\
                       0 & f_2'&0 &  f_4'  \\
                     \end{array}
                   \right), \ ~\hbox{ and }\  ~ f=\int_0^z\check{f}dz=\left(
                     \begin{array}{cccc}
                       0&  0 & 0 & 0 \\
                       0 & f_2 &  0&  f_4 \\
                     \end{array}
                   \right).\]
Since $f\check{f}=0$, we obtain $g=0$. And a straightforward computation yields
\[d=(d_{ij})=\left(
                                    \begin{array}{cc}
                                      1 +  |f_4|^2& 0 \\
                                     0& 1\\
                                    \end{array}
                                  \right),\ ~\hbox{ and }\  ~  d^{-1}=\left(
                                    \begin{array}{cc}
                                       \frac{1}{|d|} & 0 \\
                                      0 &  1\\
                                    \end{array}
                                  \right),\ ~\hbox{ with  }\  ~ |d|= 1+|f_4|^2.\]
 Since $W_0\in G(8,\mathbb{C})$, we derive
 \[a^t=Jd^{-1}J=\left(
                                    \begin{array}{cc}
                                      1  & 0\\
                                      0 &   \frac{1}{|d|}\\
                                    \end{array}
                                  \right)= a.\]
By \eqref{eq-iso-mc:1B}, we have
\[ u= \frac{1}{|d|} \left(
                                     \begin{array}{cccc}
                                     0&0&0&0 \\
                                     0 & f_2 & 0  &  f_4 \\
                                     \end{array}
                                   \right).\]
Substituting these into \eqref{eq-iso-mc:1C}, we obtain
\begin{equation*}q=(q_{ij}) =\left(
                                                                            \begin{array}{cccc}
                                                                               \frac{1}{|d|}  & -\frac{\bar{f}_2f_4 }{|d|} & 0&  0 \\
                                                                              0 & 1 & 0  &  0\\
                                                                              -\frac{f_2\bar{f}_4}{|d|} & \frac{|f_2f_4|^2}{|d|}&  1 & \bar{f}_2f_4 \\
                                                                              0 &  f_2\bar{f}_4  & 0  &  |d| \\
                                                                            \end{array}
                                                                          \right).\end{equation*}
Set
\[l_0=\left(
    \begin{array}{cccc}
     \frac{1}{\sqrt{|d|}} & -\frac{\bar{f}_2f_4}{\sqrt{|d|}}  & 0  & 0 \\
     0  & 1& 0 &  0 \\
    0 & 0 & 1 & \bar{f}_2f_4 \\
     0 &0  &0  &   \sqrt{|d|}\\\
    \end{array}
  \right),\ ~\hbox{ with }\  ~  l_0^{-1}=\left(
    \begin{array}{cccc}
      \sqrt{|d|}  &  \bar{f}_2f_4 & 0  & 0 \\
     0  & 1& 0 &0 \\
    0 & 0 & 1 &  -\frac{\bar{f}_2f_4}{\sqrt{|d|}} \\
     0 &0  &0  &   \frac{1}{\sqrt{|d|}}\\\
    \end{array}
  \right)
. \]
 It is straightforward to verify that $l_0\in G(4,\mathbb{C})$ and $q=\check{J}_4\bar{l}_0^t\check{J}_4l_0.$
By \eqref{eq-iso-frame}, we have
\[\left(
                               \begin{array}{cccc}
                                  \mathrm{f}_{13} & \mathrm{f}_{14} &\mathrm{f}_{15}     &  \mathrm{f}_{16}   \\
                                  \mathrm{f}_{23} & \mathrm{f}_{24} &\mathrm{f}_{25}     &  \mathrm{f}_{26}   \\
                                  \mathrm{f}_{33} & \mathrm{f}_{34} &\mathrm{f}_{35}     &  \mathrm{f}_{36}   \\
                                  \mathrm{f}_{43} & \mathrm{f}_{44} &\mathrm{f}_{45}     &  \mathrm{f}_{46}   \\
                                  \mathrm{f}_{53} & \mathrm{f}_{54} &\mathrm{f}_{55}     &  \mathrm{f}_{56}   \\
                                  \mathrm{f}_{63} & \mathrm{f}_{64} &\mathrm{f}_{65}     &  \mathrm{f}_{66}   \\
                                  \mathrm{f}_{73} & \mathrm{f}_{74} &\mathrm{f}_{75}     &  \mathrm{f}_{76}   \\
                                  \mathrm{f}_{83} & \mathrm{f}_{84} &\mathrm{f}_{85}     &  \mathrm{f}_{86}   \\
                               \end{array}
                             \right)= \left(
        \begin{array}{ccc}
            f+gJ\bar{u}S_0 \\
             I-f^{\sharp}J\bar{u}S_0 \\
            J \bar{u}S_0  \\
        \end{array}
      \right)l_0^{-1}=\left(
                                             \begin{array}{cccc}
                       0&  0 & 0 & 0 \\
                       0  & f_2& 0 & \frac{f_4}{\sqrt{|d|}}\\
                       \frac{1}{\sqrt{|d|}} & 0 &0 & 0 \\
                       0 & 1  & 0& 0 \\
                       -\frac{f_2\bar{f}_4}{\sqrt{|d|}} & -|f_2|^2 & 1 & -\frac{\bar{f}_2f_4}{\sqrt{|d|}} \\
                       0 & 0 & 0 & \frac{1}{\sqrt{|d|}} \\
                       \frac{\bar{f}_4}{\sqrt{|d|}} &  \bar{f}_2 & 0 & 0   \\
                       0&0&0&0 \\
                                             \end{array}
                                           \right).\]

{\em Step 2: Computation  of the Maurer-Cartan form of $F$.}
Applying Lemma \ref{lemma-iso4}, the $l_1\check{f}l_0^{-1}$ part of the Maurer-Cartan form of $\check{F}$ is of the form
\[l_1\check{f}l_0^{-1}=\left(
                     \begin{array}{cccc}
                       0&  0 & 0 & 0 \\
                       0& \frac{f_2'}{\sqrt{|d|}} & 0  &  \frac{f_4'}{|d|} \\
                     \end{array}
                   \right)\ \Longrightarrow \ B_1=\frac{1}{2}\left(
                     \begin{array}{cccc}
                      -\frac{if_2'}{\sqrt{|d|}} &  \frac{f_2'}{\sqrt{|d|}} & 0 & 0 \\
                      \frac{if_2'}{\sqrt{|d|}}&  -\frac{f_2'}{\sqrt{|d|}} & 0 & 0 \\
                       \frac{f_4'}{|d|} &  \frac{if_4'}{|d|} & 0 & 0  \\
                       \frac{if_4'}{|d|} & - \frac{f_4'}{|d|} & 0 & 0  \\
                     \end{array}
                   \right).\]
\ \\

{\em Step 3: Computation  of  $Y$.}
     Here we follow the discussions in Section 4.3.  First from the Maurer-Cartan form we have
\[D_z\phi_1=-\frac{if_2'}{2\sqrt{|d|}}(\psi_1+i\psi_2),\ \ \ D_z\phi_2=-\frac{if_2'}{2\sqrt{|d|}}(\psi_1+i\psi_2),\ \]
\[ D_z\phi_3=-\frac{f_4'}{2|d|}(\psi_1+i\psi_2),\ \ \ D_z\phi_4=-\frac{if_4'}{2|d|}(\psi_1+i\psi_2),\]
with $D$ the normal connection. Set  $E_1=\phi_1-\phi_2,\ \hat{E}_1=\phi_1+\phi_2,\ E_2=\phi_3-i\phi_4.$
Assume that
\[Y=\hat{E}_1+\bar\mu E_2+\mu \bar{E}_2+|\mu|^2E_1\]
for some $\mu$. We have that
$D_zY=\left(\frac{-if_2'}{\sqrt{|d|}}-\bar\mu\frac{f_4'}{|d|}\right)(\psi_1+i\psi_2).$
So $D_zY=0$ if and only if $\bar\mu=-\frac{if_2'\sqrt{|d|}}{f_4'}$. This yields \eqref{eq-min-iso}. \hfill$\Box$

\begin{remark} Note that the above Iwasawa decomposition only blows up at the poles of $f_2$ and $f_4$, showing that the decomposition of  the corresponding  harmonic map does not cross  the boundary of a Iwasawa big cell.
\end{remark}
\subsection{Proof of Theorem \ref{thm-example}}
 The proof is the same as the above one, except that the computation will be more tedious. Moreover, we also need to check the immersion properties of $y$. Hence we have four steps:

  1. Computation of the first four columns of $F$;

  2. Computation of the Maurer-Cartan form of $F$;

  3. Computation of $Y$;

  4. Computation of metric of $Y$.

{\em Step 1: Computation  of the first four columns of $F$.}
Since $\eta$ is of the form stated in \eqref{eq-example-np}, by \eqref{eq-B1-to-f}, it is easy to derive that
 \[\check{\mathcal{P}}(\eta)=\lambda^{-1}\left(
                                \begin{array}{ccc}
                                  0 & \check{f}  & 0 \\
                                  0 & 0 & -\check{f}^{\sharp} \\
                                  0 & 0 & 0 \\
                                \end{array}
                              \right)dz\]
                              with
 \[\check{f}=\left(
                     \begin{array}{cccc}
                       0&  0 & -1 & -z \\
                       2 & -2z &0&  0 \\
                     \end{array}
                   \right), \ ~\hbox{ and }\  ~ f=\int_0^z\check{f}dz=\left(
                     \begin{array}{cccc}
                       0&  0 & -z & -\frac{z^2}{2} \\
                       2z & -z^2 &0&  0 \\
                     \end{array}
                   \right).\]
Note now
\[g=-\int_0^z f (\check{f}^{\sharp})dz=\frac{z^3}{3}\left(
       \begin{array}{cc}
        -1 & 0 \\
         0 & 1 \\
       \end{array}
     \right).\]
  Set $r=\sqrt{|z|^2}$.
By the first equation of \eqref{eq-iso-mc:1}, we have
\[d=(d_{ij})=\left(
                                    \begin{array}{cc}
                                      1+4r^2+\frac{r^6}{9} & r^2\bar{z} \\
                                      r^2z & 1+\frac{r^4}{4}+\frac{r^6}{9} \\
                                    \end{array}
                                  \right),\ d^{-1}=\frac{1}{|d|}\left(
                                    \begin{array}{cc}
                                      1+\frac{r^4}{4}+\frac{r^6}{9} & -r^2\bar{z} \\
                                      -r^2z &  1+4r^2+\frac{r^6}{9}\\
                                    \end{array}
                                  \right).\]
with
$|d|=(1+4r^2+\frac{r^4}{4}+\frac{r^6}{9})(1+\frac{r^6}{9}).$
 Since $W_0\in G(8,\mathbb{C})$, we have
 \[a =(Jd^{-1}J)^t=\frac{1}{|d|}\left(
                                    \begin{array}{cc}
                                      1+4r^2+\frac{r^6}{9} & -r^2\bar{z} \\
                                      -r^2z & 1+\frac{r^4}{4}+\frac{r^6}{9} \\
                                    \end{array}
                                  \right).
 \]
    It is easy to verify that
    \[a=\bar{l_1}^tl_1, \ \ \hbox{ with }\ \ l_1=\left(
       \begin{array}{cc}
         \frac{\sqrt{d_{11}}}{\sqrt{|d|}} & -\frac{d_{12}}{\sqrt{|d|}\sqrt{d_{11}}} \\
         0& \frac{1}{\sqrt{d_{11}}} \\
       \end{array}
     \right)=\frac{1}{\sqrt{|d|}}\left(
                                    \begin{array}{cc}
                                      \sqrt{1+4r^2+\frac{r^6}{9}} & \frac{-r^2\bar{z}}{\sqrt{1+4r^2+\frac{r^6}{9}}} \\
                                      0 &  \frac{1}{\sqrt{1+4r^2+\frac{r^6}{9}}} \\
                                    \end{array}
                                  \right).\]
Moreover, by the second equation of \eqref{eq-iso-mc:1}, one computes
\[u^{\sharp}=\frac{1}{|d|} \left(
                                     \begin{array}{cc}
                                     \frac{r^2z^3}{2}\left(1+\frac{4r^2}{3}\right)   & -\frac{z^2}{2}\left(1+\frac{4r^2}{3}\right)\left(1+4r^2+\frac{r^6}{9}\right) \\
                                       \frac{r^2z^2}{3}\left(2-\frac{r^6}{9}-\frac{r^4}{4}\right) & -z\left(1+4r^2-\frac{2r^6}{9}\right) \\
                                       -z^2\left(1+\frac{r^4}{4}+\frac{4r^6}{9}\right) & \frac{r^4z}{3}\left(4+4r^2+\frac{r^6}{9}\right) \\
                                       2z\left(1-\frac{r^4}{12}\right)\left(1+\frac{r^4}{4}+\frac{r^6}{9}\right) & -2r^4\left(1-\frac{r^4}{12}\right) \\
                                     \end{array}
                                   \right).\]
Moreover, substituting these  into  \eqref{eq-iso-mc:1C}, we obtain $q=(q_{ij})$ with
\begin{equation*} \begin{split}
                                       &|d|q_{11}=(1+4r^2-\frac{2r^6}{9})^2,\ |d|q_{12}=|d|\bar{q}_{31}=-2r^2z(1-\frac{r^4}{12})(1+4r^2-\frac{2r^6}{9}),\\
                                       &|d|q_{13}=|d|\bar{q}_{21}=-\frac{r^2z}{2}(1+\frac{4r^2}{3})(1+4r^2-\frac{2r^6}{9}),\ |d|q_{14}=|d|\bar{q}_{41}=-r^4z^2(1+\frac{4r^2}{3})(1-\frac{r^4}{12}),\\
                                     \end{split}
                                  \end{equation*}
\begin{equation*} \begin{split}
 &|d|q_{22}=|d|\bar{q}_{33}= (1+4r^2-\frac{2r^6}{9})(1+\frac{r^4}{4}+\frac{4r^6}{9}),\ \
                                        |d|q_{23}=\frac{r^6}{4}(1+\frac{4r^2}{3})^2 ,\\
                                       &|d|q_{24}=|d|\bar{q}_{43}= \frac{r^2z}{2}(1+\frac{4r^2}{3})(1+\frac{r^4}{4}+\frac{4r^6}{9}),\ \
                                        |d|q_{32}= 4r^6(1-\frac{r^4}{12})^2,\\
                                       &|d|q_{34}=|d|\bar{q}_{42}=2r^2z(1-\frac{r^4}{12})(1+\frac{r^4}{4}+\frac{4r^6}{9}) ,\ \
                                        |d|q_{44}=(1+\frac{r^4}{4}+\frac{4r^6}{9})^2.\\
                                     \end{split}
                                  \end{equation*}
Let $l_0=(l_{jl})$ be an upper triangular matrix of the form
\[l_0=\left(
    \begin{array}{cccc}
      \frac{ (1+4r^2-\frac{2r^6}{9})}{\sqrt{|d|}}  &  -\frac{2r^2z(1-\frac{r^4}{12})}{\sqrt{|d|}}  & -\frac{r^2z(1+\frac{4r^2}{3})}{2\sqrt{|d|}}  & -\frac{r^4z^2(1+\frac{4r^2}{3})(1-\frac{r^4}{12})}{\sqrt{|d|}(1+4r^2-\frac{2r^6}{9})} \\
     0  &\sqrt{|d|}& 0 & \frac{r^2z(1+\frac{4r^2}{3})\sqrt{|d|}}{2(1+4r^2-\frac{2r^6}{9})} \\
    0 & 0 &  \frac{1}{\sqrt{|d|}}  & \frac{2r^2z(1-\frac{r^4}{12})}{\sqrt{|d|}(1+4r^2-\frac{2r^6}{9})}\\
     0 &0  &0  &   \frac{\sqrt{|d|}}{ (1+4r^2-\frac{2r^6}{9})}\\
    \end{array}
  \right).\]
It is straightforward to check that $q=\check{J}_4\bar{l}_0^t\check{J}_4l_0$ and
 \[ l_0^{-1}=\left(
    \begin{array}{cccc}
      \frac{\sqrt{|d|}}{ (1+4r^2-\frac{2r^6}{9})}  &
       \frac{2r^2z(1-\frac{r^4}{12})}{\sqrt{|d|}(1+4r^2-\frac{2r^6}{9})} &
       \frac{r^2z(1+\frac{4r^2}{3})\sqrt{|d|}}{2(1+4r^2-\frac{2r^6}{9})} & -\frac{r^4z^2(1+\frac{4r^2}{3})(1-\frac{r^4}{12})}{\sqrt{|d|}(1+4r^2-\frac{2r^6}{9})} \\
      0 & \frac{1}{\sqrt{|d|}} & 0 &  -\frac{r^2z(1+\frac{4r^2}{3})}{2\sqrt{|d|}} \\
    0 & 0 &  \sqrt{|d|}  & -\frac{2r^2z(1-\frac{r^4}{12})}{\sqrt{|d|}}\\
    0  & 0 & 0 &  \frac{ (1+4r^2-\frac{2r^6}{9})}{\sqrt{|d|}}\\
    \end{array}
  \right).\]
 Assume that
 \[\check{F}=(\mathrm{f}_{jl})=\frac{1}{|d|^{\frac{3}{2}}(1+4r^2-\frac{2r^6}{9})}(\hat{\mathrm{f}}_{jl}).\] By \eqref{eq-iso-frame},
     we compute
   \begin{equation}\label{eq-false-iwa}
  \left(\hat{\mathrm{f}}_{jl}\right)=
    \left(
     \begin{array}{cccccccc}
      -\frac{r^8}{6\lambda}(1+\frac{4r^2}{3})  & \frac{zr^4}{3\lambda} &  -\frac{z(1+4r^2)|d|}{\lambda} & -\frac{z^2 \hat{\mathrm{f}}_{25}}{2} \\
        \frac{2z}{\lambda}\tilde{\mathrm{f}}_{23}  & -\frac{z^2}{\lambda} &  \frac{2z^2r^2|d|}{3\lambda} &  \frac{z^3r^2(1+\frac{4r^2}{3})}{2\lambda} \\
        \hat{\mathrm{f}}_{24}  &  2z r^2 &  -\frac{4zr^4|d|}{3} &  -z^2r^4(1+\frac{4r^2}{3}) \\
       -\frac{r^2\bar{z}}{2}\hat{\mathrm{f}}_{25}    & 1+4r^2+\frac{r^6}{9} & -r^2(1+4r^2)|d| &  -\frac{zr^2}{2}\tilde{\mathrm{f}}_{25} \\
      \frac{r^6\bar{z}}{2}(1+\frac{4r^2}{3})   & -r^4 & (1+4r^2+\frac{4r^6}{9})|d|   & \frac{r^6\bar{z}}{2}(1+\frac{4r^2}{3})  \\
      -r^4\bar{z}^2(1+\frac{4r^2}{3})   & 2\bar{z}r^2  &   -\frac{4\bar{z}r^4|d|}{3} &  \hat{\mathrm{f}}_{24}\\
      \frac{\lambda r^2\bar{z}^3}{2}(1+\frac{4r^2}{3})  &-\lambda\bar{z}  &  \frac{2\lambda \bar{z}^2r^2|d|}{3}  &   2\lambda\bar{z} \tilde{\mathrm{f}}_{23}  \\
      -\frac{\lambda\bar{z}^2}{2 }\tilde{\mathrm{f}}_{25} & \frac{\lambda\bar{z}r^4}{3} &  -\lambda \bar{z}(1+4r^2)|d|  & -\frac{\lambda r^8}{6}(1+\frac{4r^2}{3})  \\
     \end{array}
   \right)
 \end{equation}
 with $1\leq j \leq 8, 3\leq l\leq6$, and \begin{equation*} \begin{split}
                                       &\tilde{\mathrm{f}}_{23}=1+4r^2+\frac{r^4}{6}-\frac{2r^6}{9}+\frac{r^{10}}{54},\ \tilde{\mathrm{f}}_{25}= (1+\frac{4r^2}{3} ) (1+4r^2+\frac{r^6}{9} ),\\
                                       &\hat{\mathrm{f}}_{24}=1+4r^2-\frac{10r^6}{9}-\frac{8r^8}{9}-\frac{2r^{12}}{81}.\\
                                     \end{split}
                                  \end{equation*}

{\em Step 2: Computation  of the Maurer-Cartan form of $F$.}
By \eqref{eq-iso-m-c1} and \eqref{eq-iso-m-c2} of Lemma \ref{lemma-iso4}, the Maurer-Cartan form of $\check{F}$ has the expression
\begin{equation*}
\begin{split}
 l_1\check{f}l_0^{-1} & =\frac{1}{\sqrt{|d|}}\cdot l_1\cdot \left(
                                     \begin{array}{cccc}
  0&0&-|d|&-z(1+2r^2-\frac{r^6}{18})\\
  \frac{2|d|}{(1+4r^2-\frac{2r^6}{9})} & \frac{-2z(1+2r^2-\frac{r^6}{18})}{(1+4r^2-\frac{2r^6}{9})} &   \frac{zr^2(1+\frac{4r^2}{3})}{(1+4r^2-\frac{2r^6}{9})} & \frac{z^2r^2(1+\frac{4r^2}{3})(1+2r^2-\frac{r^6}{18})}{(1+4r^2-\frac{2r^6}{9})} \\
                                     \end{array}
                                   \right) \\
    & =\sqrt{|d|}\left(
       \begin{array}{cccc}
         \tilde{w}_1 & -\tilde{w}_1\rho & \tilde{w}_2 & \tilde{w}_2\rho \\
         \hat{w}_1 & -\hat{w}_1\rho & \hat{w}_2 & \hat{w}_2\rho \\
       \end{array}
     \right),
\end{split}
\end{equation*}
with
\[\rho=\frac{z(1+2r^2-\frac{r^6}{18})}{|d|},~~ ~l_1=\frac{1}{\sqrt{|d|}}\left(
                                    \begin{array}{cc}
                                      \sqrt{1+4r^2+\frac{r^6}{9}} & \frac{-r^2\bar{z}}{\sqrt{1+4r^2+\frac{r^6}{9}}} \\
                                      0 &  \frac{1}{\sqrt{1+4r^2+\frac{r^6}{9}}} \\
                                    \end{array}
                                  \right)=\left(
        \begin{array}{cc}
          \hat{l}_{11} & \hat{l}_{12}  \\
          0 & \hat{l}_{22}  \\
        \end{array}
      \right),\]
and \begin{equation*} \begin{split}
&\tilde{w}_1= \frac{2\hat{l}_{12}}{1+4r^2-\frac{2r^6}{9}},\ \ \tilde{w}_2=\frac{\hat{l}_{12}zr^2(1+\frac{4r^2}{3})}{1+4r^2-\frac{2r^6}{9}}-\hat{l}_{11},\ \hat{w}_1=  \frac{2\hat{l}_{22}}{1+4r^2-\frac{2r^6}{9}},\ \hat{w}_2= \frac{ \hat{l}_{22} zr^2(1+\frac{4r^2}{3})}{1+4r^2-\frac{2r^6}{9}}.\\
\end{split}
                                  \end{equation*}
Transforming back to $\mathfrak{so}(1,7,\C)$, we derive
\[\check{B}_1=\left(
                \begin{array}{cccc}
                  i(\hat{w}_2+\hat{w}_1\rho) & -(\hat{w}_2+\hat{w}_1\rho)  &  i(\tilde{w}_2+\tilde{w}_1\rho) & -(\tilde{w}_2+\tilde{w}_1\rho)  \\
                  i(\hat{w}_2-\hat{w}_1\rho) & -(\hat{w}_2-\hat{w}_1\rho)  &  i(\tilde{w}_2-\tilde{w}_1\rho) & -(\tilde{w}_2-\tilde{w}_1\rho)  \\
                  (\hat{w}_2\rho-\hat{w}_1) & i(\hat{w}_2\rho-\hat{w}_1)  &  (\tilde{w}_2\rho-\tilde{w}_1) & i(\tilde{w}_2\rho-\tilde{w}_1)  \\
                  i(\hat{w}_2\rho+\hat{w}_1) & -(\hat{w}_2\rho+\hat{w}_1)  &  i(\tilde{w}_2\rho+\tilde{w}_1) & -(\tilde{w}_2\rho+\tilde{w}_1)  \\
                \end{array}
              \right)=(h_1,ih_1,h_3,ih_3).
\]
\vspace{5mm}

{\em Step 3: Computation  of  $Y$.}
     Here we follow the discussion  in Section 4.3.
It is easy to verify that $h_1$ and $h_3$ can be expressed as a (functional) linear combination of
\[(1,1,-i\rho,\rho)^t\ \hbox{ and }\  (\rho,-\rho,i,1)^t.\]
Therefore one obtains easily that
$\check{\rho}_1=-i\rho$
is the unique solution to \eqref{eq-iso-derivative2}. Substituting $\check{\rho}_1$ into \eqref{eq-y}, we obtain
\begin{equation*}
\begin{split}
 [Y]  & =(1+|\rho|^2)\phi_1+(1-|\rho|^2)\phi_2-i(\rho-\bar{\rho})\phi_3+(\rho+\bar{\rho})\phi_4] \\
    & =[(\phi_1+\phi_2) +|\rho|^2(\phi_1-\phi_2)+i\bar\rho(\phi_3-i\phi_4)-i\rho(\phi_3+i\phi_4)].
\end{split}
\end{equation*}
Then by \eqref{eq-false-iwa}, \eqref{eq-iso-frame2} we have that
\begin{equation}
[Y]=\left[\frac{(1+4r^2-\frac{2r^6}{9})}{|d|^{\frac{5}{2}}}\left(
                          \begin{array}{c}
                            \left(1+r^2+\frac{5r^4}{4}+\frac{4r^6}{9}+\frac{r^8}{36}\right) \\
                            \left(1-r^2-\frac{3r^4}{4}+\frac{4r^6}{9}-\frac{r^8}{36}\right) \\
                            -i\left(z- \bar{z})(1+\frac{r^6}{9})\right) \\
                            \left(z+\bar{z})(1+\frac{r^6}{9})\right) \\
                            -i\left((\lambda^{-1}z^2-\lambda \bar{z}^2)(1-\frac{r^4}{12})\right) \\
                            \left((\lambda^{-1}z^2+\lambda \bar{z}^2)(1-\frac{r^4}{12})\right) \\
                            -i\frac{r^2}{2}(\lambda^{-1}h-\lambda \bar{h})(1+\frac{4r^2}{3}) \\
                            \frac{r^2}{2} (\lambda^{-1}h+\lambda \bar{h})(1+\frac{4r^2}{3})  \\
                          \end{array}
                        \right)\right].
\end{equation}

\vspace{5mm}
{\em Step 4: $[Y]$ being a global immersion.}
Let $x_{\lambda}$ be of the form \eqref{example1}. Then $x_{\lambda}:S^2\rightarrow S^6$ is well-defined on $S^2$ with $Y$ as its lift. Since
\[|x_{\lambda z}|^2|dz|^2=\frac{2+8r^2+\frac{r^4}{2}+\frac{4r^6}{9}+\frac{8r^8}{9}+\frac{r^{10}}{18}+\frac{2r^{12}}{81}}{\left(1+r^2+\frac{5r^4}{4}+\frac{4r^6}{9}+\frac{r^8}{36}\right)^2}|dz|^2,\]
$x$ has no branch point at $z\in \mathbb{C}$.
As to $\infty$, set $\tilde{z}=\frac{1}{z}$ and $\tilde{r}=\sqrt{|\tilde{z}|}$, we have that  $|x_{\lambda\tilde{z}}|^2|d\tilde{z}|^2=32|d\tilde{z}|^2$ at the point $\tilde{z}=0$, since
\[|x_{{\lambda}\tilde{z}}|^2|d\tilde{z}|^2=\frac{2\tilde{r}^{12}+8\tilde{r}^{10}+\frac{\tilde{r}^8}{2}+\frac{4\tilde{r}^6}{9}+\frac{8\tilde{r}^4}{9}+\frac{\tilde{r}^{2}}{18} +\frac{2}{81}}{\left(\tilde{r}^8+\tilde{r}^6+\frac{5\tilde{r}^4}{4}+\frac{4\tilde{r}^2}{9}+\frac{1}{36}\right)^2}|d\tilde{z}|^2,\]
     \hfill$\Box$

\begin{remark} Note that in the above Iwasawa decomposition there exists a circle $1+4r^2-\frac{2r^6}{9}=0$ such that the frame \eqref{eq-false-iwa} obtained from the Iwasawa decomposition blows up. However, this blowing up can be avoided by a change of frames and hence the corresponding harmonic map is in fact globally well defined. This also means that the decomposition of the corresponding  harmonic map does not  cross  the boundary of an Iwasawa big cell as the other examples. From this one  can  expect  that there are more interesting phenomena happening for harmonic maps into non-compact symmetric spaces comparing with the compact cases (see also \cite{B-R-S}, \cite{Ke1}).
\end{remark}

\

{\large \bf Acknowledgements}\ \ The author is thankful to Professor Josef Dorfmeister, Professor Changping Wang and Professor Xiang Ma for their suggestions and encouragement.   This work is supported by the Project 11201340 of NSFC.

{\small
\def\refname{Reference}

}

{\small
Peng Wang

Department of Mathematics

Tongji University, Siping Road 1239

Shanghai, 200092, P. R. China

{\em E-mail address}: {netwangpeng@tongji.edu.cn}
}
\end{document}